\documentclass[reqno]{amsart}
\def\TeXRoot{.}
\input{\TeXRoot/preamble/0_input.sty}
\input{./settings.sty}
\begin{document}

\maketitle

\begin{abstract}
  In this paper, we introduce a Grothendieck topology on the category of totally bounded metric spaces and develop a theory of stacks with respect to this topology. 
  We further define the fine moduli stack of compact metric spaces and prove that its coarse moduli space is isometric to the Gromov--Hausdorff space.
\end{abstract}

\tableofcontents

\section*{Introduction}

The space \(\topGH\) of isometry classes of compact metric spaces equipped with the Gromov--Hausdorff distance plays an important role in metric geometry as a coarse moduli space of compact metric spaces.
However, since there exist nontrivial families parametrizing mutually isometric objects, the Gromov--Hausdorff space does not serve as a fine moduli space.
In contrast, in algebraic geometry, the theory of algebraic stacks was introduced by Deligne--Mumford and later generalized and developed by Artin. 
This theory provides fine moduli for a wide range of moduli problems in the form of algebraic stacks.
Today, algebraic stacks are widely recognized as the correct geometric framework for discussing moduli problems in algebraic geometry.
In this paper, we construct a fine moduli stack of compact metric spaces by adapting the theory of moduli stacks to metric spaces.
We show that the theory of stacks provides an effective framework for moduli problems in metric geometry as well.
This theory can be regarded as a partial realization of what was suggested by Gromov in \cite[{3.11~{\small\(\frac{3}{4}_+\)}}]{Gromov}.


In this paper, in order to develop the theory of stacks appropriately in the context of metric geometry, we introduce a Grothendieck topology---what we call the \(\lhftop\) topology \confer{\cref{def-lhf-top}}---on the category \(\tbMet\) of totally bounded metric spaces.
By viewing proper \Hfibs as families of compact metric spaces parametrized by metric spaces, we then define a stack \(\stGH\) over \(\tbMet_{\lhftop}\) \confer{\cref{def-stGH}}.
We further show that \(\stGH\) can be presented as a quotient stack \confer{\cref{prop-stGH-quot}}.
We call such a stack a naive metric stack of quotient type \confer{\cref{def-naive-metric-stack} and \cref{def-of-quot-type}}.
Finally, we prove that \(\stGH\) satisfies the appropriate properties of a fine moduli in the category of naive metric stacks of quotient type \confer{\cref{prop-fine-moduli-GH}}.


Our main result is the following:

\begin{theoremA}
  The moduli stack \(\stGH\) of compact metric spaces \confer{\cref{def-stGH}} is a naive metric stack of quotient type \confer{\cref{prop-stGH-quot}}.
  Moreover, for any naive metric stack \(\mcX\) of quotient type, pulling back the universal object \(\stGH_1\to\stGH\) induces a one-to-one correspondence between representable proper \Hfibs \(\mcP\to \mcX\) \confer{\cref{def-rep-prop-Hfib}} of stacks in groupoids over \(\tbMet_{\lhftop}\) and morphisms \(\mcX\to \stGH\) of stacks in groupoids over \(\tbMet_{\lhftop}\) \confer{\cref{prop-fine-moduli-GH}}.
\end{theoremA}

\subsection*{Outline}

In \cref{section: category Met}, we study a category-theoretic property of the category of extended pseudometric spaces \(\epMet\) and the category of presheaves on \(\tbMet\). 
In \cref{section: Hfib}, we study basic properties of morphisms of metric spaces, called \textbf{\Hfibs}, from the viewpoint of metric geometry, and establish the foundational results needed in this paper.
In \cref{section: lhftop}, we define a Grothendieck topology---what we call the \(\lhftop\) topology \confer{\cref{def-lhf-top}}---on the category of totally bounded metric spaces using the notion of \textbf{\morphLHF} which is obtained by localizing the notion of \Hfib.
In \cref{section: naive metric stack}, we develop the basic theory of stacks over the site \(\tbMet_{\lhftop}\), define good moduli spaces as morphisms to extended pseudometric spaces satisfying an appropriate universality property, and investigate their structural features.
In \cref{section: GH}, we define the moduli stack of \(n\)-pointed compact metric spaces \(\stGH_n\), study its fundamental properties, and finally prove the main theorem.

\subsection*{Acknowledgements}

First, I would like to express my deep gratitude to my friend M. Nakagawa, who has strongly encouraged my mathematical activities. 
Without his significant encouragement, this paper would not have been completed.
I am also grateful to Y. Ishiki for helpful discussions on metric-theoretic properties of submetries.
Finally, I would like to thank all my friends who have encouraged my pursuit of mathematical research on my own initiative.

\subsection*{Disclosure of AI Assistance}

During the preparation of this paper, the author made use of large language models as auxiliary tools. 
In particular, ChatGPT was used repeatedly for English proofreading and stylistic refinement, and, in non-technical parts of the paper such as the introduction, for discussion of the overall exposition and organization. 
These uses substantially improved the efficiency of the writing process.
In addition, Gemini was consulted to identify the existence of prior work related to local submetries and to generate a typographical symbol \qedcatsymbol\ for the end of proofs.
All mathematical content, results, proofs, and final decisions regarding exposition are entirely the responsibility of the author.




\subsection*{Notations and Conventions}

The notation \(\R\) denotes the set of real numbers, equipped with its structure as complete ordered field and the ordinary distance \(d_{\R}(x,y) \dfn |x-y|\) (\(\forall x,y\in \R\)).
The notation \(\N, \Z, \Q\) denotes the subsets of \(\R\) consisting of non-negative integers, integers, and rational numbers, respectively.
By a slight abuse of notation, we endow \(\N\), \(\Z\), and \(\Q\) with the additional structures induced from \(\R\).

We write \(\bR\dfn \R \cup \{-\infty, \infty\}\).
We regard \(\bR\) as a complete totally ordered set with respect to the usual order on \(\R\), extended by declaring \((-\infty,\infty) = (\min\bR,\max\bR)\).
For any \(a,b\in \bR\), we write \(a\lor b \dfn \max\{a,b\}\) and \(a\land b\dfn \min\{a,b\}\).
We write \(\bRz\dfn\left\{ x \geq 0\,\middle|\, x\in \bR\right\}\) and \(\Rz = \bRz\cap \R\).
We write \(\bRp \dfn \bRz\setminus \{0\}\) and \(\Rp \dfn \bRp \setminus \{\infty\}\).
We regard \(\bRz\) as a totally ordered commutative monoid under the usual addition of real numbers, with the unique absorbing element \(\infty\in \bRz\), and with the order induced from \(\bR\).

\subsection*{Categories}

The notation \(\sfSet\) will be used to denote the (big) category of all (small) sets.
The notation \(\Top\) will be used to denote the (big) category of all (small) topological spaces.

Let \(\mcC, \mcD\) be categories.
We write \(\Fun(\mcC, \mcD)\) for the category of functors and natural transformations.
We write \(\PSh(\mcC) = \hat{\mcC}\dfn \Fun(\mcC^{\op}, \sfSet)\).
If \(\mcC\) is locally small, then we write \(h^{\mcC}_{\bullet}:\mcC\to \PSh(\mcC)\) for the Yoneda embedding.
By a slight abuse of notations, we often omit the superscript \(\mcC\) and write \(h_{\bullet}\) simply.

\subsection*{Metric Spaces}

An \textbf{extended pseudometric space} \(X = (\uSet{X}, d_X)\) consists of a set \(\uSet{X}\) together with a function \(d_X:\uSet{X}\times\uSet{X}\to \bRz\) satisfying \(d_X(x,y) + d_X(y,z) \geq d_X(x,z) = d_X(z,x)\) for all \(x,y,z\in \uSet{X}\).
We refer to \(\uSet{X}\) as the \textbf{underlying set} of \(X\), and to \(d_X\) as the \textbf{distance} on \(X\).
By a slight abuse of notation, we write \(x\in X\) to mean \(x\in \uSet{X}\).
A \textbf{pseudometric space} is an extended pseudometric space \(X\) such that \(\infty\not\in\im(d_X)\).
An \textbf{extended metric space} is an extended pseudometric space \(X\) such that for any \(x,y\in X\), the condition \(d_X(x,y)=0\) implies \(x=y\).
In particular, a \textbf{metric space} is both an extended metric space and pseudometric space.

Let \(X, Y\) be extended pseudometric spaces.
A \textbf{morphism} of extended pseudometric spaces \(f:X\to Y\) is a map \(f:\uSet{X}\to\uSet{Y}\) such that for any \(x,x'\in X\), it holds that \(d_X(x,x')\geq d_Y(f(x),f(x'))\), that is, a morphism is precisely a \(1\)-Lipschitz map.

We write \(\PEMet\) for the category of (small) extended pseudometric spaces and 1-Lipschitz maps.
We write \(\EMet\subset \PEMet\) for the full subcategory consisting of the extended metric spaces.
We write \(\PMet\subset \PEMet\) for the full subcategory consisting of the pseudometric spaces.
We write \(\Met\dfn \PMet\cap \EMet \subset \PEMet\) for the full subcategory consisting of the metric spaces.
We write \(\tbMet\) for the full subcategory of \(\Met\) consisting of the totally bounded metric spaces.

\begin{remark*}
  Since a compact metric space is of cardinality at most continuum, and the completion of a totally bounded metric space is compact, it follows that a totally bounded metric space is of cardinality at most continuum. 
  Hence, in particular, \(\tbMet\) is essentially small (i.e., there exists a small set of objects of \(\tbMet\) such that any object of \(\tbMet\) is isomorphic to an object included in this set).
\end{remark*}
\section{The Category of Metric Spaces}
\label{section: category Met}

In this section, we first provide, for the reader’s convenience, a proof of the well-known fact that the category \(\epMet\) of extended pseudometric spaces is complete and cocomplete \confer{\cref{lem: lim and colim in PEMet}}.
As a consequence, the inclusion functor \(\tbMet\inj \epMet\) admits a left Kan extension along the Yoneda embedding, yielding an adjoint pair \(\PSh(\tbMet) \leftrightarrows \epMet\). 
\[\begin{tikzcd}
  {\PSh(\tbMet)} \ar[rd, xshift=3mm, yshift=3mm] &[1cm]  \\
  {\tbMet} \ar[u] \ar[r, hook] & {\PEMet}. \ar[ul, xshift=4mm, yshift=4mm]
\end{tikzcd}\]
We then verify that the right adjoint functor \(\epMet\to \PSh(\tbMet)\) is  fully faithful, and investigate the structure of the objects obtained as images of presheaves under the left Kan extension \(\PSh(\tbMet)\to \epMet\).

\begin{notation}
  We write \(\uSet{(-)}:\PEMet\to \sfSet\) for the functor which maps a metric space \((-)\) to its underlying set.
\end{notation}

\begin{remark}\label{rmk-uSet-lr-adj}
  The functor \(\uSet{(-)}:\PEMet\to \sfSet\) admits both a left adjoint and a right adjoint.
  The left adjoint assigns to each set the extended pseudometric space in which the distance between any two distinct points is \(\infty\).
  The right adjoint assigns to each set the extended pseudometric space in which the distance between any two distinct points is \(0\).
  Consequently, for any diagram \(F:I\to \PEMet\), the natural morphisms \(\uSet{\lim F(-)}\to \lim \uSet{F(-)}\) and \(\colim \uSet{F(-)} \to \uSet{\colim F(-)}\) are isomorphisms.
\end{remark}

\begin{lemma}\label{cor: epi mono in PEMet}
  Let \(f:X\to Y\) be a morphism in \(\PEMet\).
  Then, the following assertions hold:
  \begin{assertion}
    \item \label{cor: epi mono in PEMet ass: mono}
    \(f\) is a monomorphism in \(\PEMet\) if and only if \(\uSet{f}:\uSet{X}\to \uSet{Y}\) is injective.
    \item \label{cor: epi mono in PEMet ass: epi}
    \(f\) is an epimorphism in \(\PEMet\) if and only if \(\uSet{f}:\uSet{X}\to \uSet{Y}\) is surjective.
  \end{assertion}
\end{lemma}

\begin{proof}
  Both assertions in \cref{cor: epi mono in PEMet} follow immediately from \cref{rmk-uSet-lr-adj}.
\end{proof}

Based on \cref{rmk-uSet-lr-adj}, we can compute explicitly limits and colimits in \(\PEMet\) as follows:

\begin{lemma}\label{lem: lim and colim in PEMet}
  Let \(X_{\bullet}:I\to \PEMet\) be a diagram in \(\PEMet\).
  Then, the following assertions hold.
  \begin{assertion}
    \item \label{lem: lim and colim in PEMet ass: lim}
    For any \((x_i)_{i\in I}, (x_i')_{i\in I}\in \lim_{i\in I} \uSet{X_i}\), define
    \[d_{\lim}((x_i)_{i\in I}, (x_i')_{i\in I})\dfn \sup_{i\in I}d_{X_i}(x_i, x_i').\]
    Then, the pair \((\lim_{i\in I} \uSet{X_i}, d_{\lim})\) is an extended pseudometric space, and it is naturally isomorphic to the limit of the diagram \(X_{\bullet}\).
    Moreover, if each \(X_i\) is an extended metric space, then the limit \((\lim_{i\in I} \uSet{X_i}, d_{\lim})\) is also an extended metric space.
    \item \label{lem: lim and colim in PEMet ass: colim}
    For any \(i\in I\), let \(\alpha_i:\uSet{X_i}\to \colim_{i\in I}\uSet{X_i}\) denote the canonical map.
    For any \(x,x'\in \colim_{i\in I} \uSet{X_i}\), define
    \begin{align*}
      &d_{\colim}(x,x') \\
      &\dfn \inf\left\{\sum_{k=0}^N d_{X_{i_k}}(x_{i_k},x_{i_k}')\,\middle|\,\begin{aligned}
        &N\in\N, i_0,\cdots,i_N\in I, \\
        &x_{i_k}, x_{i_k}'\in X_{i_k} \ (\forall k\in \{0,\cdots, N\}) \\
        &\alpha_{i_0}(x_{i_0})=x, \ \alpha_{i_N}(x'_{i_N})=x', \\
        &\alpha_{i_k}(x_{i_k}') = \alpha_{i_{k+1}}(x_{i_{k+1}}) \ (\forall k\in \{0,\cdots, N-1\})
      \end{aligned}\right\}.
    \end{align*}
    Then, the pair \(X\dfn (\colim_{i\in I} \uSet{X_i}, d_{\colim})\) is an extended pseudometric space, and \(X\) is naturally isomorphic to the colimit of the diagram \(X_{\bullet}\).
  \end{assertion}
\end{lemma}

\begin{proof}
  \Cref{lem: lim and colim in PEMet ass: lim} follows immediately from elementary discussion about extended pseudometric spaces.

  Next, we verify \cref{lem: lim and colim in PEMet ass: colim}.
  It is straightforward to verify that \(X = (\colim_{i\in I} \uSet{X_i}, d_{\colim})\) is an extended pseudometric space, and the family of canonical maps \((\alpha_i: \uSet{X_i}\to \uSet{X})_{i\in I}\) forms a functorial family of morphisms in \(\PEMet\).
  To verify a universal property of the colimit, let \(Y\in \PEMet\) be an object, \((f_i: X_i\to Y)_{i\in I}\) a functorial family of morphisms in \(\PEMet\), \(x,x'\in X\) elements, and \(\ep > 0\) a positive real number.
  Then, by the definition of \(d_{\colim}\), there exist an integer \(N\in \N\), indices \(i_0,\cdots,i_N\in I\), and elements \(x_{i_k}, x_{i_k}'\in X_{i_k}\) such that \(\alpha_{i_0}(x_{i_0})=x\), \(\alpha_{i_N}(x'_{i_N})=x'\), \(\alpha_{i_k}(x_{i_k}') = \alpha_{i_{k+1}}(x_{i_{k+1}})\) (\(\forall k\in\{0,\cdots, N-1\}\)), and
  \[d_X(x,x') \geq -\ep + \sum_{k=0}^N d_{X_{i_k}}(x_{i_k}, x_{i_k}').\]
  Write \(f: \uSet{X} = \colim_{i\in I} \uSet{X_i}\to \uSet{Y}\) for the induced map \confer{\cref{rmk-uSet-lr-adj}}.
  Then, it holds that \(f(x) = f(\alpha_{i_0}(x_{i_0})) = f_{i_0}(x_{i_0})\), and \(f(x') = f(\alpha_{i_N}(x_{i_N}')) = f_{i_N}(x_{i_N}')\).
  Hence, it holds that
  \begin{align*}
    d_X(x,x') &\geq -\ep + \sum_{k=0}^N d_{X_{i_k}}(x_{i_k}, x_{i_k}') \\
    &\geq -\ep + \sum_{k=0}^N d_Y(f_{i_k}(x_{i_k}), f_{i_k}(x_{i_k}')) \\
    &\geq -\ep + d_Y(f(x), f(x')).
  \end{align*}
  By letting \(\ep \to 0\), we conclude that \(d_X(x,x') \geq d_Y(f(x), f(x'))\).
  In particular, \(f\) is a morphism in \(\PEMet\).
  Finally, since \(\uSet{X} = \colim_{i\in I}\uSet{X_i}\), if there exist morphisms \(g,g':X\to Y\) in \(\PEMet\) such that \(g\circ \alpha_i = g'\circ \alpha_i\) for any \(i\in I\), then \(g = g'\).
  Thus, \(X\) is naturally isomorphic to the colimit of the diagram \(X_{\bullet}\).
  This completes the proof of \cref{lem: lim and colim in PEMet}.
\end{proof}

\begin{remark}
  The metric identification \confer{\cite[Example 2]{Jardine-Met-Rips}} yields a left adjoint functor \((-)^*: \PEMet\to \EMet\) to the inclusion functor \(\EMet\subset \PEMet\).
  Consequently, \(\EMet\) is complete and cocomplete.
\end{remark}

Next, we investigate the adjoint pair \(\PSh(\tbMet)\leftrightarrows \epMet\) obtained from the left Kan extension of the inclusion functor \(\tbMet\subset \epMet\) along the Yoneda embedding \(\tbMet\inj \PSh(\tbMet)\). 

\begin{notation}\label{def: univ adj}
  \
  \begin{enumerate}
    \item
    We write \(h_{(-)}\dfn \Hom_{\tbMet}(-, X):\tbMet\inj \PSh(\tbMet)\) for the Yoneda embedding. 
    By a slight abuse of notation, we also write \(h_{(-)}:\epMet\to \PSh(\tbMet)\) for the functor sending an extended pseudometric space \(X\) to the presheaf \(h_X\dfn \Hom_{\epMet}(-, X)\restriction_{\tbMet}:\tbMet^{\op}\subset \epMet^{\op}\to\sfSet\).
    \item
    We write \(\metRe:\PSh(\tbMet)\to \EPMet\) for the left Kan extension of the inclusion functor \(\tbMet\inj \PEMet\) along the Yoneda embedding \(\tbMet\to \PSh(\tbMet)\).
    \item \label{def: univ adj enumi: slice cat}
    Let \(\mcX\in \PSh(\tbMet)\) be an object.
    We write \(\tbMet_{/\mcX}\) for the comma category \confer{\cite[{\href{https://alg-d.com/math/kan_extension/comma.pdf}{\JPstr{コンマ圏}}}]{alg-d}} obtained as the lax pull-back of the diagram of categories \(*\xto{\mcX}\PSh(\tbMet)\xgets{h} \tbMet\):
    \[\begin{tikzcd}
      {*} \ar[r, "{\mcX}"] &[0.5cm] {\PSh(\tbMet)} \ar[rd, "{\metRe}", ""{name=L, below}, xshift=3mm, yshift=3mm] &[1cm]  \\
      {\tbMet_{/\mcX}} \ar[u] \ar[r] &
      {\tbMet}
        \ar[u, "{h}", hook] \ar[r, hook]
        \ar[lu, Rightarrow, shorten=1cm]
        \ar[Rightarrow, to=L, shorten=5mm] &
      {\PEMet},
    \end{tikzcd}\]
    where \(*\) denotes the discrete category with a single object, and \(\mcX: *\to \PSh(\tbMet)\) is the functor sending the unique object of \(*\) to \(\mcX\).
  \end{enumerate}
\end{notation}

Note that by \cite[{\href{https://alg-d.com/math/kan_extension/kan_extension.pdf}{\JPstr{Kan 拡張}}, \JPstr{系 30}}]{alg-d}, for any object \(\mcX\in\PSh(\tbMet)\), the natural morphism
\[\colim(\tbMet/\mcX \to \tbMet \xinj{h} \PSh(\tbMet)) \to \mcX\]
is an isomorphism in \(\PSh(\tbMet)\).
In particular, the natural morphism
\[\colim(\tbMet/\mcX \to \tbMet \xinj{i} \PEMet) \to \metRe(\mcX)\]
is an isomorphism in \(\PEMet\).
Moreover, by \confer{\cite[{\href{https://alg-d.com/math/kan_extension/kan_extension.pdf}{\JPstr{Kan 拡張}}, \JPstr{系 28}, \JPstr{定理 45}}]{alg-d}}, the functor \(h:\epMet\to \PSh(\tbMet)\) is a right adjoint of \(\mu\).

\begin{notation}
  We write \(*\dfn\{0\}\subset \R\) for the singleton equipped with the induced distance.
  For any \(r\in \bRp\), we write \(2_r\dfn \{0,r\}\subset \bRz\) for the two-point metric space equipped with the induced distance.
  For any \(r,s\in\bRp\) with \(r > s\), we write \(\iota_{rs}:2_r\to 2_s\) for the bijective morphism in \(\epMet\) such that \(\iota_{rs}(0) = 0\).
  For any \(r\in\bRp\), we write \(\iota_{0r}:*\to 2_r\) (resp. \(\iota_r:*\to 2_r\)) for the morphism in \(\epMet\) satisfying \(\iota_{0r}(0) = 0\) (resp. \(\iota_r(0) = r\)). 
  By a slight abuse of notation, when the meaning is clear from the context, we often omit the subscript \(r\) in \(\iota_{0r}\) and simply write \(\iota_0\).
\end{notation}

\begin{example}\label{eg: 2-r colim}
  Write \(F:\Rp\to \EMet\) for the diagram determined by \(F(r)\dfn 2_r\) and \(F(r\geq s)\dfn \iota_{rs}\), where \(\Rp\) is regarded as a category via its usual order.
  Then, the colimit of the composite \(F:\Rp\to \EMet \subset \PEMet\) has cardinality \(2\), and the distance of the distinct points is \(0\) \confer{\cite[Example 4]{Jardine-Met-Rips}}.
  Thus, the colimit of the diagram \(F:\Rp\to \EMet\) is isomorphic to the terminal object \(*\in \EMet\).
\end{example}

\begin{lemma}\label{lem: univ adj PEMet}
  The following assertions hold:
  \begin{assertion}
    \item \label{lem: univ adj PEMet ass: uset}
    For any object \(\mcX\in\PSh(\tbMet)\), the natural map \(u:\mcX(*)\to \uSet{\metRe(\mcX)}\) is bijective.
    Moreover, for any points \(x,x'\in \metRe(\mcX)\) and any \(\ep > 0\), there exist an integer \(N\in \N\), elements \(r_i\in \bRp\) (\(i\in \{0,\cdots,N\}\)), and morphisms \(f_i:h_{2_{r_i}}\to \mcX\) (\(i\in \{0,\cdots,N\}\)) such that the following conditions hold:
    \begin{condition}
      \item \label{lem: univ adj PEMet ass: uset condi: sum}
      \(\sum_{i=0}^Nr_i \leq d_{\metRe(\mcX)}(x,x') + \ep\).
      \item \label{lem: univ adj PEMet ass: uset condi: term}
      It holds that \(u(f_0(*)(\iota_0:*\to 2_{r_0})) = x\), and \(u(f_N(*)(\iota_{r_N}:*\to 2_{r_N})) = x'\).
      \item \label{lem: univ adj PEMet ass: uset condi: chain}
      For any \(i\in\{1,\cdots,N\}\), it holds that \(f_{i-1}(*)(\iota_{r_{i-1}}:*\to 2_{r_{i-1}})=f_i(*)(\iota_0:*\to 2_{r_i})\).
    \end{condition}
    \item \label{lem: univ adj PEMet ass: counit}
    The natural transformation \(\metRe\circ h \to \id_{\PEMet}\) arising from the adjoint pair \(\metRe:\PSh(\tbMet)\leftrightarrows \epMet:h\) is an isomorphism.
    \item \label{lem: univ adj PEMet ass: ff}
    The functor \(h: \PEMet\to \PSh(\tbMet)\) is fully faithful; that is, for any extended pseudometric spaces \(X,Y\in\PEMet\) and any morphism \(\varphi: \Hom_{\PEMet}(-, X)\to \Hom_{\PEMet}(-, Y)\) in \(\PSh(\tbMet)\), there exists a unique 1-Lipschitz map \(f:X\to Y\) such that \(\varphi = f\circ (-)\).
  \end{assertion}
\end{lemma}

\begin{proof}
  \Cref{lem: univ adj PEMet ass: uset} follows immediately from \cref{lem: lim and colim in PEMet} \ref{lem: lim and colim in PEMet ass: colim}, together with the fact that the natural morphism
  \[\colim(\tbMet/\mcX \to \tbMet \xinj{i} \PEMet) \to \metRe(\mcX)\]
  is an isomorphism.
  \Cref{lem: univ adj PEMet ass: counit} follows immediately from \cref{lem: univ adj PEMet ass: uset}.
  \Cref{lem: univ adj PEMet ass: ff} follows immediately from \Cref{lem: univ adj PEMet ass: counit}, together with a well-known result from elementary category theory \confer{\cite[{\href{https://alg-d.com/math/kan_extension/adjoint.pdf}{\JPstr{随伴関手}}, \JPstr{定理 18} (6)}]{alg-d}}.
  This completes the proof of \cref{lem: univ adj PEMet}.
\end{proof}

\begin{notation}\label{notation: PEmet subset PSh}
  \
  \begin{enumerate}
    \item \label{notation: PEmet subset PSh enumi: full sub}
    By a slight abuse of notation, we regard \(\PEMet\) as a full subcategory of \(\PSh(\tbMet)\) via the fully faithful embedding \(\PEMet\inj\PSh(\tbMet)\) induced by the left Kan extension of the Yoneda embedding \(\tbMet\to \PSh(\tbMet)\) along the inclusion functor \(\tbMet\inj \PEMet\) (cf. \cref{lem: univ adj PEMet} \ref{lem: univ adj PEMet ass: ff}).
    Thus, for any object \(X\in \PEMet\), we identify \(X\) with the presheaf \(\Hom_{\PEMet}(-,X)\restriction_{\tbMet}:\tbMet^{\op}\subset \PEMet\to \sfSet\).
    \item \label{notation: PEmet subset PSh enumi: uSet}
    By a slight abuse of notation, for any object \(\mcX\in\PSh(\tbMet)\), we identify the underlying set of \(\metRe(\mcX)\) with the set \(\mcX(*)\) \confer{\cref{lem: univ adj PEMet} \ref{lem: univ adj PEMet ass: uset}}.
    \item \label{notation: PEmet subset PSh enumi: metRe}
    For any object \(\mcX\in\PSh(\tbMet)\), we write \(\metRe_{\mcX}:\mcX\to \metRe(\mcX)\) for the natural morphism arising from the adjoint pair \((\metRe,\PEMet\subset\PSh(\tbMet))\).
  \end{enumerate}
\end{notation}

\begin{definition}\label{def: repble}
  Let \(B\in \epMet\) be an object.
  We say that a presheaf \(\mcF\in\PSh(\tbMet_{/B})\) on \(\tbMet_{/B}\) is \textbf{representable by an extended pseudometric space} if there exists an object \(X\in\epMet_{/B}\) together with an isomorphism \(\mcF\cong X\) in \(\PSh(\tbMet_{/B})\). 
\end{definition}

\begin{remark}
  Since any surjective morphism is an epimorphism in \(\PEMet\) \confer{\cref{cor: epi mono in PEMet} \ref{cor: epi mono in PEMet ass: epi}}, there exists an epimorphism \(f\) that is not a coequalizer of its kernel pair.
  Thus, by Giraud's theorem \confer{\cite[Theorem 0.45]{Johnstone}}, \(\PEMet\) is not a Grothendieck topos.
  In particular, \(\metRe\) does not preserve finite limits.
\end{remark}

\section{The Hausdorff Distance}
\label{section: Hfib}

In this section, we study basic properties of morphisms of metric spaces, called \textbf{\Hfibs}, from the viewpoint of metric geometry, and establish the foundational results needed in this paper.
Submetries have been used effectively in differential geometry to study Riemannian submersions; from the perspective of metric geometry, however, it is natural to regard them as analogues of flat families in algebraic geometry.

We also show that the functor parametrizing compact subspaces of a given metric space is representable by the space of compact subsets equipped with the Hausdorff distance.
This result can be viewed as a metric-geometric analogue of the representability of the Hilbert functor in algebraic geometry.

\begin{definition}\label{dfn-Hfib}
  \
  \begin{enumerate}
    \item \label{dfn-Hfib-Hdist}
    Let \(X\in\epMet\) be an object.
    For any subsets \(F_0, F_1\subset X\), we write
    \[\dH_X(F_0, F_1) \dfn \max\left\{\sup_{x_0\in F_0}\inf_{x_1\in F_1}d(x_0,x_1), \sup_{x_1\in F_1} \inf_{x_0\in F_0}d(x_0,x_1)\right\} \in \R_{\geq 0}\cup \{\infty\}.\]
    We say that \(\dH_X\) is the \textbf{Hausdorff distance} on \(X\).
    \item \label{dfn-Hfib-Hfib}
    Let \(f:X\to B\) be a morphism in \(\epMet\).
    We say that \(f\) is a \textbf{\Hfib} if \(f\) is surjective, and, moreover, for any points \(b_0, b_1\in B\), it holds that \(\dH_X(f^{-1}(b_0), f^{-1}(b_1)) \leq d_B(b_0, b_1) < \infty\).
  \end{enumerate}
\end{definition}

\begin{remark}\label{rmk-Hfib}
  \
  \begin{enumerate}
    \item \label{rmk-Hfib-dist}
    Let \(X\in\epMet\) be an object.
    Then, for any subset \(F\subset X\), we have \(\dH_X(F,\overline{F}) = 0\).
    In particular, \(\dH_X\) is not a metric on the set of subsets of \(X\).
    However, if \(X\) is a metric space, then \(\dH_X\) gives a distance on the set of non-empty bounded closed subsets of \(X\) \confer{\cite[{Problem 4.5.23}]{Engelk}}.
    \item \label{rmk-Hfib-surj}
    Let \(f:X\to B\) be a morphism in \(\epMet\).
    For any metric space \(X\) and any subset \(F\subset X\), we have \(\dH_X(\emptyset, F) = \infty\).
    In particular, if \(B\in \pMet\) and \(f:X\to B\) satisfies that \(\dH_X(f^{-1}(b_0), f^{-1}(b_1)) \leq d_B(b_0, b_1)\) for every \(b_0, b_1\in B\), then \(f\) is surjective.
  \end{enumerate}
\end{remark}

\begin{notation}
  Let \(X\) be an extended pseudometric space, \(x\in X\) a point, and \(r\in\bR\).
  We define
  \begin{align*}
    \ball_X(x,r) &\dfn \left\{ x'\in X \,\middle|\, d_X(x,x') < r \right\}, \ \text{and} \\
    \clball_X(x,r) &\dfn \left\{ x'\in X \,\middle|\, d_X(x,x') \leq r \right\}.
  \end{align*}
  In particular, \(\ball_X(x,r) = \emptyset\) for all \(r \leq 0\) and \(x\in X\), and \(\clball_X(x, \infty) = X\) for all \(x\in X\).
  Moreover, \(X\) is an extended metric space if and only if \(\clball_X(x, 0) = \{x\}\) for every \(x\in X\).
  Note that \(X\) is a pseudometric space if and only if \(\ball_X(x, \infty) = X\). 
  We call \(\ball_X(x,r)\) the open ball of radius \(r\) centered at \(x\).
\end{notation}

\begin{lemma}\label{prop-Hfib-char}
  Let \(f:X\to B\) be a morphism in \(\epMet\) and \(b_0,b_1\in B\) points.
  Then, the following assertions are equivalent:
  \begin{assertion}
    \item \label{prop-Hfib-char-Hfib}
    \(\dH_X(f^{-1}(b_0), f^{-1}(b_1)) = d_B(b_0, b_1)\).
    \item \label{prop-Hfib-char-inf}
    For any points \(x_0\in f^{-1}(b_0)\) and \(x_1\in f^{-1}(b_1)\), 
    \[\inf_{\tilde{x}_1\in f^{-1}(b_1)}d_X(x_0, \tilde{x}_1) = \inf_{\tilde{x}_0\in f^{-1}(b_0)}d_X(\tilde{x}_0, x_1) = d_B(b_0,b_1).\]
    \item \label{prop-Hfib-char-ineq}
    For any positive real number \(\ep>0\) and 
    any point \(x_0\in f^{-1}(b_0)\), 
    there exists a point \(x_1\in f^{-1}(b_1)\) such that
    \(d_X(x_0, x_1) \leq d_B(f(x_0), b_1) + \ep\).
    \item \label{prop-Hfib-char-lift}
    For any positive real number \(\ep > 0\) and any commutative diagram
    \[\begin{tikzcd}
      {*} \ar[r, "x"] \ar[d, "{\iota_0}"'] & {X} \ar[d, "f"] \\
      {2_{\ep + d_B(b_0, b_1)}} \ar[r, "g"] & {B}
    \end{tikzcd}\]
    in \(\epMet\) such that \(g(0) = b_0\) and \(g(r) = b_1\), there exists a morphism \(h:2_{\ep + d_B(b_0, b_1)}\to X\) in \(\epMet\) such that \(x = h \circ \iota_0\) and \(g = f\circ h\).
  \end{assertion}
\end{lemma}

\begin{proof}
  First, we prove the equivalence ``\ref{prop-Hfib-char-Hfib}\(\Leftrightarrow\)\ref{prop-Hfib-char-inf}''.
  For any points \(x_0\in f^{-1}(b_0)\) and \(x_1\in f^{-1}(b_1)\), we have
  \begin{align*}
    \dH_X(f^{-1}(b_0), f^{-1}(b_1))
    &= \max\left\{\sup_{\tilde{x}_0\in f^{-1}(b_0)}\inf_{\tilde{x}_1\in f^{-1}(b_1)}d_X(\tilde{x}_0, \tilde{x}_1), \sup_{\tilde{x}_1\in f^{-1}(b_1)}\inf_{\tilde{x}_0\in f^{-1}(b_0)}d_X(\tilde{x}_0, \tilde{x}_1)\right\} \\
    &\geq \min\left\{\inf_{\tilde{x}_1\in f^{-1}(b_1)}d_X(x_0, \tilde{x}_1), \inf_{\tilde{x}_0\in f^{-1}(b_0)}d_X(\tilde{x}_0, x_1)\right\} \\
    &\!\underset{\scriptscriptstyle\text{\hypertarget{prop-Hfib-char-Hfib-ineq}{(\(\bigstar\))}}}{\geq}\!
    \min\left\{\inf_{\tilde{x}_1\in f^{-1}(b_1)}d_B(f(x_0), f(\tilde{x}_1)), \inf_{\tilde{x}_0\in f^{-1}(b_0)}d_B(f(\tilde{x}_0), f(x_1))\right\} \\
    &= d_B(b_0, b_1),
  \end{align*}
  where the inequality \hyperlink{prop-Hfib-char-Hfib-ineq}{(\(\bigstar\))} follows from the definition of a morphism in \(\epMet\).
  This shows that ``\ref{prop-Hfib-char-Hfib}\(\Leftrightarrow\)\ref{prop-Hfib-char-inf}''.
  The equivalence ``\ref{prop-Hfib-char-inf}\(\Leftrightarrow\)\ref{prop-Hfib-char-ineq}'' follows immediately from the definition of the infimum and the surjectivity of \(f\).
  The equivalence ``\ref{prop-Hfib-char-ineq}\(\Leftrightarrow\)\ref{prop-Hfib-char-lift}'' follows immediately from the definition of a morphism in \(\epMet\). 
  This completes the proof of \cref{prop-Hfib-char}.
\end{proof}

\begin{lemma}\label{prop-Hfib-char-ball}
  Let \(f:X\to B\) be a surjective morphism in \(\epMet\).
  Then, \(f\) is a \Hfib if and only if for any point \(x\in X\) and any positive real number \(r > 0\), it holds that \(\ball_B(f(x),r) = f(\ball_X(x,r))\).
\end{lemma}

\begin{proof}
  First, we prove necessity.
  Assume that \(f\) is a \Hfib.
  Let \(x\in X\) be a point and \(r > 0\) a positive real number.
  Since \(f\) is a morphism in \(\epMet\), it holds that \(f(\ball_X(x,r))\subset \ball_B(f(x),r)\).
  Let \(b\in \ball_B(f(x), r)\) be a point.
  Then, there exists a positive real number \(\ep > 0\) such that \(d_B(f(x), b) + \ep < r\).
  Hence, by \cref{prop-Hfib-char} ``\ref{prop-Hfib-char-Hfib}\(\Leftrightarrow\)\ref{prop-Hfib-char-ineq}'', there exists a point \(x'\in f^{-1}(b)\) such that \(d_X(x,x') \leq d_B(f(x),b) + \ep < r\).
  In particular, \(x'\in \ball_X(x,r)\).
  Thus, it holds that \(\ball_B(f(x),r)\subset f(\ball_X(x,r))\).
  This completes the proof of necessity.
  Sufficiency follows immediately from \cref{prop-Hfib-char} ``\ref{prop-Hfib-char-Hfib}\(\Leftrightarrow\)\ref{prop-Hfib-char-lift}'', together with the fact that the existence of the lifting \(h:2_r\to X\) is equivalent to the condition that \(\ball_X(x, r) \cap f^{-1}(g(r))\neq \emptyset\).
  This completes the proof of \cref{prop-Hfib-char-ball}.
\end{proof}

\begin{corollary}\label{prop-Hfib-basic}
  Let \(f:X\to B\) be a morphism in \(\epMet\).
  Then, the following assertions hold:
  \begin{assertion}
    \item \label{prop-Hfib-basic-open}
    If \(f\) is a \Hfib, then \(f\) is an open map.
    \item \label{prop-Hfib-basic-comp}
    Let \(g:Y\to X\) be a morphism in \(\epMet\).
    Assume that \(g\) is a \Hfib.
    Then, \(f\) is a \Hfib if and only if \(f\circ g\) is a \Hfib.
    \item \label{prop-Hfib-basic-dense-imt}
    Let \(p:U\to X\) be a morphism in \(\epMet\) whose image \(\im(p)\subset X\) is dense.
    Assume that \(f\circ p\) is a \Hfib.
    Then, \(f\) is a \Hfib.
    \item \label{prop-Hfib-basic-bc}
    Let \(B'\to B\) be a morphism in \(\epMet\).
    Assume that \(f\) is a \Hfib.
    Then, the natural projection \(X\times_B B'\to B'\) is a \Hfib.
  \end{assertion}
\end{corollary}

\begin{proof}
  \Cref{prop-Hfib-basic-open,prop-Hfib-basic-comp} follow immediately from \cref{prop-Hfib-char-ball}.

  Next, we prove \cref{prop-Hfib-basic-dense-imt}.
  Let \(x\in X\) be a point, \(b'\in B\) a point, and \(\ep > 0\) a positive real number.
  Since \(\im(p)\subset X\) is dense, there exists a point \(u\in U\) such that \(d_X(x, p(u)) < \ep/3\).
  Since \(f\circ p\) is a \Hfib, it follows from \cref{prop-Hfib-char} ``\ref{prop-Hfib-char-Hfib}\(\Leftrightarrow\)\ref{prop-Hfib-char-ineq}'' that there exists a point \(u'\in (f\circ p)^{-1}(b')\) such that
  \[d_U(u,u') \leq d_B(f(p(u)), f(p(u'))) + \ep/3.\]
  Then,
  \begin{align*}
    d_B(f(x), f(p(u'))) + \ep
    &> d_B(f(x), f(p(u'))) + d_X(x,p(u)) + 2\ep/3 \\
    &\geq d_B(f(x),f(p(u')))+d_B(f(x),f(p(u)))+2\ep/3 \\
    &\geq
    d_B(f(p(u)),f(p(u')))+2\ep/3  \\
    &\geq
    d_U(u,u')+\ep/3 \\
    &\geq
    d_X(p(u),p(u'))+\ep/3 \\
    &> d_X(p(u),p(u'))+d_X(x,p(u))
    \geq d_X(x,p(u')).
  \end{align*}
  Since \(p(u')\in f^{-1}(b')\), it follows from \cref{prop-Hfib-char} ``\ref{prop-Hfib-char-Hfib}\(\Leftrightarrow\)\ref{prop-Hfib-char-ineq}'' that \(f\) is a \Hfib.
  This completes the proof of \cref{prop-Hfib-basic-dense-imt}.

  \Cref{prop-Hfib-basic-bc} follows immediately from \cref{prop-Hfib-char} ``\ref{prop-Hfib-char-Hfib}\(\Leftrightarrow\)\ref{prop-Hfib-char-lift}''.
  This completes the proof of \cref{prop-Hfib-basic}.
\end{proof}

\begin{example}
  Let \(M\) be a complete Riemannian manifold, \(G\) a Lie group, \(\pi:P\to M\) a principal \(G\)-bundle, and \(\omega\) a connection on \(P\).
  Then, by the Hopf-Rinow theorem, for any point \(p\in P\) and any point \(x\in M\), there exists a minimizing geodesic \(\gamma:[0,1]\to M\) such that \(\gamma(0)=\pi(p)\) and \(\gamma(1)=x\).
  By lifting \(\gamma\) to \(P\) horizontally, we obtain a geodesic \(\tilde{\gamma}:[0,1]\to P\) (with respect to the bundle metric on \(P\) determined by the connection \(\omega\) and the Riemannian metrics on \(M\) and \(G\)) starting at \(p\).
  Write \(p'\dfn \tilde{\gamma}(1)\).
  Then, by \cite[Lemma 26.11 (1) (2)]{Michor}, it holds that \(\gamma = \pi\circ\tilde{\gamma}\) and that \(d_P(p, p') = d_M(\pi(p), x)\).
  Thus, it follows from \cref{prop-Hfib-char} ``\ref{prop-Hfib-char-ineq}\(\Rightarrow\)\ref{prop-Hfib-char-Hfib}'' that \(\pi\) is a \Hfib.
  Note that this result holds regardless of whether the curvature of \(\omega\) vanishes or not.
\end{example}

\begin{example}\label{example-X/G}
  Let \(G\) be a group and \((X, d_X)\) a metric space equipped with a (left) \(G\)-action.
  Write \(p:X\to X/G\) for the natural projection.
  Assume that the following conditions hold:
  \begin{condition}
    \item \label{example-X/G-Ginv}
    The distance \(d_X\) on \(X\) is \(G\)-invariant, i.e., for any \(g\in G\) and any \(x_1,x_2\in G\), it holds that \(d_X(x_1,x_2) = d_X(gx_1, gx_2)\).
    \item \label{example-X/G-cl-orb}
    The action \(G\act X\) has closed orbits, i.e., for any element \(x\in X\), the orbit \(G\cdot x\subset X\) is closed.
  \end{condition}
  Then, by \cref{example-X/G-Ginv}, for any \(x_1, x_2\in X\), it holds that
  \[\dH_X(G\cdot x_1, G\cdot x_2) = \inf_{g\in G}d_X(x_1, gx_2) \leq d_X(x_1, x_2).\]
  Hence, the Hausdorff distance \(\dH_X\) on \(X\) defines a distance on \(X/G\) such that the natural projection \(p:X\to X/G\) is a \Hfib.
\end{example}

\begin{example}
  Write \(p: \R\to \{\Q, \R\setminus\Q\}\) for the map defined as \(p(x)\dfn\) \(\Q\) if \(x\in \Q\) else \(\R\setminus \Q\). 
  The pseudodistance on \(\{\Q,\R\setminus\Q\}\) is defined as \(d(\Q,\R\setminus \Q)=0\). 
  Then, \(p\) is a \Hfib.
\end{example}

\begin{proposition}\label{prop-Hfib-completion}
  Let \(f:X\to B\) be a morphism in \(\Met\).
  Assume that \(f\) is a \Hfib.
  Then, the completion \(\hat{f}: \hat{X}\to \hat{B}\) of \(f\) is a \Hfib.
  In particular, \(\hat{f}\) is surjective.
\end{proposition}

\begin{proof}
  Let \(x\in \hat{X}\) be a point, \(b'\in\hat{B}\) a point, and \(\ep > 0\) a positive real number.
  Since \(X\subset \hat{X}\) is dense, there exists a point \(x_0\in X\) such that
  \begin{condition}
    \item \label{prop-Hfib-completion-ep/8}
    \(d_{\hat{X}}(x_0, x) < \ep/8\).
  \end{condition}
  Moreover, there exists a Cauchy sequence \((b_n')_{n\in \N}\) in \(B\) such that \(b' = \lim_{n\to\infty}b_n'\), and the following condition holds:
  \begin{condition}[start=2]
    \item \label{prop-Hfib-completion-s-ep}
    For any \(n\in \N\), it holds that \(\sum_{k\geq n}d_B(b_k',b_{k+1}') < \ep/2^{n+3}\).
  \end{condition}
  Since \(f\) is a \Hfib, it follows from \cref{prop-Hfib-char} ``\ref{prop-Hfib-char-Hfib}\(\Leftrightarrow\)\ref{prop-Hfib-char-ineq}'' that there exists a point \(x_0'\in f^{-1}(b_0')\) such that
  \begin{condition}[start=3]
    \item \label{prop-Hfib-completion-x0}
    \(d_X(x_0, x_0') < d_B(f(x_0), b_0') + \ep/4\).
  \end{condition}
  By induction on \(k\), it follows from \cref{prop-Hfib-char} ``\ref{prop-Hfib-char-Hfib}\(\Leftrightarrow\)\ref{prop-Hfib-char-ineq}'' that for any \(k\in \N\), there exists a point \(x_{k+1}'\in f^{-1}(b_{k+1}')\) such that
  \begin{condition}[start=4]
    \item \label{prop-Hfib-completion-xk}
    \(d_X(x_k',x_{k+1}') < d_B(b_k',b_{k+1}') + \ep/2^{k+4}\).
  \end{condition}
  Then, by \ref{prop-Hfib-completion-s-ep} and \ref{prop-Hfib-completion-xk}, for any integers \(0 \leq n < m\), it holds that
  \begin{align*}
    d_X(x_n', x_m') &\leq \sum_{n \leq k < m} d_X(x_k', x_{k+1}')
    < \sum_{n \leq k < m} (d_B(b_k', b_{k+1}') + \ep/2^{k+4}) \\
    &< \ep/2^{n+3} + \ep/2^{n+3} = \ep/2^{n+2}.
  \end{align*}
  Hence, the sequence \((x_n')_{n\in \N}\) is a Cauchy sequence in \(X\).
  Write \(x'\dfn \lim_{n\to \infty}x_n'\in \hat{X}\).
  Then, \(\hat{f}(x') = \lim_{n\to \infty}f(x_n') = \lim_{n\to \infty}b_n' = b'\) and \(d_{\hat{X}}(x_0',x') = \lim_{m\to \infty}d_{\hat{X}}(x_0', x_m') \leq \ep/2^{0+2} = \ep/4\).
  Hence, by \ref{prop-Hfib-completion-ep/8} and \ref{prop-Hfib-completion-x0}, it holds that
  \begin{align*}
    d_{\hat{X}}(x, x') &\leq d_{\hat{X}}(x, x_0) + d_X(x_0,x_0') + d_{\hat{X}}(x_0', x') \\
    &< \ep/8 + d_B(f(x_0), b_0') + \ep/4 + \ep/4 = d_B(f(x_0), b_0') + 5\ep/8.
  \end{align*}
  Moreover, by \ref{prop-Hfib-completion-ep/8} and \ref{prop-Hfib-completion-s-ep}, it holds that
  \begin{align*}
    d_B(f(x_0), b_0') + 5\ep/8 &\leq d_{\hat{B}}(f(x_0), \hat{f}(x)) + d_{\hat{B}}(\hat{f}(x), b') + d_{\hat{B}}(b', b_0') + 5\ep/8 \\
    &\leq d_{\hat{X}}(x_0,x) + d_{\hat{B}}(\hat{f}(x), b') + \sum_{n\in \N}d_B(b_n',b_{n+1}') + 5\ep/8 \\
    &< \ep/8 + d_{\hat{B}}(\hat{f}(x), b') + \ep/2^{0+3} + 5\ep/8
    < d_{\hat{B}}(\hat{f}(x), b') + \ep.
  \end{align*}
  Therefore, \(d_{\hat{X}}(x,x') < d_{\hat{B}}(\hat{f}(x), b') + \ep\). 
  By \cref{rmk-Hfib} \ref{rmk-Hfib-surj} and \cref{prop-Hfib-char} ``\ref{prop-Hfib-char-Hfib}\(\Leftrightarrow\)\ref{prop-Hfib-char-ineq}'', this shows that \(\hat{f}\) is a \Hfib.
  This completes the proof of \cref{prop-Hfib-completion}.
\end{proof}

\begin{corollary}\label{prop-Hfib-cplt-descending}
  Let \(f:X\to B\) be a morphism in \(\Met\).
  Assume that \(f\) is a \Hfib and that \(X\) is complete.
  Then, \(B\) is complete.
\end{corollary}

\begin{proof}
  By \cref{prop-Hfib-completion}, the composite \(X\tosim \hat{X}\xto{\hat{f}} \hat{B}\) is surjective.
  Hence, the natural inclusion \(B\inj \hat{B}\) is surjective.
  This implies that \(B\) is complete.
\end{proof}

\begin{lemma}\label{prop-Hfib-tot-bd}
  Let \(f:X\to B\) be a morphism in \(\Met\).
  Assume that \(B\) is totally bounded, that \(f\) is a \Hfib, and moreover that for any \(b\in B\), the fiber \(f^{-1}(b)\) is totally bounded.
  Then, \(X\) is totally bounded.
\end{lemma}

\begin{proof}
  Let \(\ep > 0\) be a positive real number.
  Since \(B\) is totally bounded, there exist an integer \(N\in \N\) and points \(b_0,\cdots ,b_N\in B\) such that \(B\subset \bigcup_{i=0}^N \ball_B(b_i,\ep/3)\).
  Since for each \(i\in \{0,\cdots ,N\}\) the fiber \(f^{-1}(b_i)\) is totally bounded, there exist an integer \(M\in \N\) and points \(a_{ij}\in X, (i\in\{0,\cdots,N\}, j\in\{0,\cdots,M\})\), such that for each \(i\in \{0,\cdots,N\}\), \(a_{ij}\in f^{-1}(b_i)\), and \(f^{-1}(b_i)=\bigcup_{j=0}^M f^{-1}(b_i)\cap \ball_X(a_{ij},\ep/3)\).
  To prove that \(X\) is totally bounded, it suffices to prove that \(X \subset \bigcup_{i=0}^N\bigcup_{j=0}^M \ball_X(a_{ij},\ep)\).
  Let \(x\in X\) be a point.
  Then, there exists an integer \(i_0\in\{0,\cdots, N\}\) such that \(f(x)\in \ball_B(b_{i_0},\ep/3)\).
  Since \(f\) is a \Hfib, there exists a point \(x_{i_0}\in f^{-1}(b_{i_0})\) such that \(d_X(x, x_{i_0}) < d_B(f(x), b_{i_0}) + \ep/3\).
  Then, there exists an integer \(j_0\in\{0,\cdots,M\}\) such that \(x_{i_0}\in f^{-1}(b_{i_0})\cap \ball_X(a_{i_0j_0}, \ep/3)\).
  Therefore,
  \[d_X(x, a_{i_0j_0}) \leq d_X(x,x_{i_0}) + d_X(x_{i_0},a_{i_0j_0})
  < d_B(f(x), b_{i_0}) + \ep/3 + \ep/3 < \ep.\]
  Thus, \(X\) is totally bounded.
  This completes the proof of \cref{prop-Hfib-tot-bd}.
\end{proof}

\begin{lemma}\label{prop: Hfib fiberwise cpt prop}
  Let \(f:X\to B\) be a morphism in \(\Met\).
  Assume that \(f\) is a \Hfib.
  Then, the underlying map of topological spaces of \(f\) is proper (i.e., a closed continuous map with compact fibers) if and only if for every \(b\in B\), the fiber \(f^{-1}(b)\) is compact.
\end{lemma}

\begin{proof}
  The necessity follows immediately from the definition of a proper map.
  Assume that for any \(b\in B\), \(f^{-1}(b)\) is compact.
  Note that any metric space is Hausdorff and compactly generated.
  Hence, by \cite[Theorem~3.7.18]{Engelk}, to prove that \(f\) is proper, it suffices to prove that \(X\) is compact whenever \(B\) is compact.
  Assume moreover that \(B\) is compact.

  Let \((a_n)_{n\in\N}\) be a Cauchy sequence in \(X\). 
  Since \(B\) is compact, the Cauchy sequence \((f(a_n))_{n\in\N}\) in \(B\) converges to a point \(b_{\infty}\in B\).
  Since \(f\) is a \Hfib, for each \(n\in \N\), there exists a point \(a'_n\in f^{-1}(b_{\infty})\) such that \(d_X(a_n,a_n') < d_B(f(a_n), b_{\infty}) + 1/2^n\).
  Let \(\ep > 0\) be a positive real number.
  There exists \(N\in \N\) such that for all \(n,m\geq N\), \(d_X(a_n, a_m) < \ep/3 - 1/2^{N-1}\).
  Then, for each \(n,m\geq N\), 
  \begin{align*}
    d_X(a'_n, a'_m) &\leq d_X(a'_n, a_n) + d_X(a_n, a_m) + d_X(a_m, a_m') \\
    &< (d_B(b_{\infty}, f(a_n)) + 1/2^n) + (\ep/3 - 1/2^{N-1}) + (d_B(b_{\infty}, f(a_m)) + 1/2^m) \\
    &\leq d_B(b_{\infty}, f(a_n)) + d_B(b_{\infty}, f(a_m)) + \ep/3 \leq \ep.
  \end{align*}
  Thus, \((a'_n)_{n\in \N}\) is a Cauchy sequence in \(f^{-1}(b_{\infty})\). 
  Since \(f^{-1}(b_{\infty})\) is a compact metric space, the sequence \((a'_n)_{n\in \N}\) converges to a unique point \(a_{\infty}\in f^{-1}(b_{\infty})\).
  Since \(d_X(a_n, a'_n) < d_B(f(a_n), b_{\infty}) + 1/2^n\) for all \(n\in\N\), this implies that the sequence \((a_n)_{n\in \N}\) converges to \(a_{\infty}\).
  In particular, \(X\) is complete.
  By \cref{prop-Hfib-tot-bd}, \(X\) is totally bounded.
  Thus, \(X\) is compact. 
  This completes the proof of \cref{prop: Hfib fiberwise cpt prop}.
\end{proof}

\begin{definition}\label{def-inCpt}
  Let \(f:X\to B\) be a morphism in \(\Met\).
  \begin{enumerate}
    \item
    We write
    \[\Cpt(X/B)\dfn \left\{A \subset X\,\middle|\,\text{\(A\) is compact, and \(\#(f(A)) = 1\)}\right\}.\]
    We regard \(\Cpt(X/B)\) as a metric space endowed with the Hausdorff distance. 
    \item
    For any object \(Z\in \tbMet_{/B}\), we write
    \[\inCpt(X/B)(Z) \dfn \left\{A\subset X\times_BZ\,\middle|\, \text{\(A\to Z\) is a proper \Hfib}\right\}.\]
    Since base change preserves the required properties, we obtain a functor
    \begin{align*}
      \inCpt(X/B): \tbMet_{/B}^{\op}&\to \sfSet, \\
      Z &\mapsto \inCpt(X/B)(Z).
    \end{align*}
  \end{enumerate}
\end{definition}

\begin{lemma}\label{prop: cpt repble}
  Let \(f:X\to B\) be a morphism in \(\Met\).
  Then, the functor \(\inCpt(X/B)\) is represented by an object \(\Cpt(X/B)\to B\) in \(\Met_{/B}\) (in the sense of \cref{def: repble}).
\end{lemma}

\begin{proof}
  By \cref{lem: univ adj PEMet}~\ref{lem: univ adj PEMet ass: uset} and the definitions of \(\inCpt(X/B)\) and \(\Cpt(X/B)\), the natural map \[\mu(\inCpt(X/B))\to \Cpt(X/B)\] is an isomorphism in \(\epMet\).
  Moreover, by the definitions of \(\inCpt(X/B)\) and \(\Cpt(X/B)\), the natural morphism of presheaves \(\eta:\inCpt(X/B)\to \Hom_{\epMet_{/B}}(-, \Cpt(X/B))\) is injective.
  Let \(T\to B\) be an object of \(\tbMet_{/B}\) and \(f:T\to \Cpt(X/B)\) a morphism in \(\epMet_{/B}\).
  Write \(A_T\dfn \left\{(x, t)\,\middle|\, x\in f(t)\right\}\subset X\times_BT\). 
  Since \(f\) is a morphism in \(\epMet_{/B}\), for any \(t_1,t_2\in T\), we have \(d_T(t_1,t_2) \geq \dH_X(f(t_1), f(t_2))\). 
  Hence, the natural projection \(A_T\to T\) is a \Hfib.
  Moreover, for any \(t\in T\), the subset \(f(t)\subset X\) is compact. 
  Hence, by \cref{prop: Hfib fiberwise cpt prop}, the natural projection \(A_T\to T\) is proper.
  Thus, the natural projection \(A_T\to T\) defines an element of \(\inCpt(X/B)(T)\) whose image under \(\eta_T:\inCpt(X/B)(T)\to \Hom_{\epMet_{/B}}(-, \Cpt(X/B))\) is equal to \(f\).
  This implies that \(\eta\) is an isomorphism.
  This completes the proof of \cref{prop: cpt repble}.
\end{proof}


\begin{corollary}\label{prop-hausdorff-dist-cplt}
  If \(X\) is a complete metric space, then \(\Cpt(X/*)\) is complete.  
\end{corollary}

\begin{proof}
  Let \(f:Z\to \Cpt(X/*)\) be a morphism in \(\Met\) from an object \(Z\in\tbMet\). 
  By \cref{prop: cpt repble}, \(f\) corresponds, via the Yoneda lemma, to a proper \Hfib \(p_f:A_f\subset X\times Z\to Z\). 
  Then, by \cref{prop-Hfib-completion}, the completion \(\hat{p}_f:\hat{A}_f\subset X\times \hat{Z}\to \hat{Z}\) is a proper \Hfib.
  Hence, by \cref{prop: cpt repble}, the proper \Hfib \(\hat{p}_f:\hat{A}\to \hat{Z}\) corresponds, via the Yoneda lemma, to a morphism \(\tilde{f}: \hat{Z}\to \Cpt(X/*)\) such that \(\tilde{f}|_Z = f\). 
  This implies that \(\Cpt(X/*)\) is complete. 
  This completes the proof of \cref{prop-hausdorff-dist-cplt}.
\end{proof}

\section{The \texorpdfstring{\(\lhftop\)}{lsm}-Topology on \texorpdfstring{\(\tbMet\)}{tbMet}}
\label{section: lhftop}

In this section, we define a Grothendieck topology on the category of totally bounded metric spaces using the notion of \textbf{\morphLHF} which is obtained by localizing the notion of \Hfib.
We call this topology the \textbf{lsm topology} \confer{\cref{def-lhf-top}}.
We further show that the lsm topology is subcanonical \confer{\cref{lem: lhf subcan}}, and that for any lsm covering, every descent datum with respect to that covering in the category of metric spaces is effective \confer{\cref{lem: eff descent}}.

\begin{definition}[{\cite[{Definition~2.7}]{{KL}}}]\label{def: lrp}
  Let \(f:X\to Y\) be a morphism in \(\Met\).
  We say that \(f\) is a \textbf{\morphLHF} if, for any point \(x\in X\), there exists a positive real number \(r>0\) such that for any point \(x'\in \ball_X(x, r)\) and any \(0 < s < r- d_X(x,x')\), it holds that \(f(\ball_X(x',s)) = \ball_Y(f(x'),s)\).
\end{definition}

\begin{lemma}\label{lem: lrp fundamental}
  Let \(f:X\to Y\) be a morphism in \(\Met\).
  Then, the following assertions hold:
  \begin{assertion}
    \item \label{lem: lrp fundamental ass: hf}
    If \(f\) is a \Hfib, then \(f\) is a \morphLHF.
    \item \label{lem: lrp fundamental ass: open imt}
    If \(f\) is an open isometric embedding, then \(f\) is a \morphLHF.
    \item \label{lem: lrp fundamental ass: open map}
    If \(f\) is a \morphLHF, then \(f\) is an open map.
    \item \label{lem: lrp fundamental ass: comp}
    Let \(g:Y\to Z\) be a morphism in \(\Met\).
    If \(f\) and \(g\) are \morphLHFs, then \(g\circ f\) is a \morphLHF.
    \item \label{lem: lrp fundamental ass: bc}
    Let \(q:Y'\to Y\) be a morphism in \(\Met\).
    Write \(f':X'\dfn X\times_YY'\to Y'\) for the natural projection.
    If \(f\) is a \morphLHF, then \(f'\) is a \morphLHF.
  \end{assertion}
\end{lemma}

\begin{proof}
  \Cref{lem: lrp fundamental ass: hf} follows immediately from \cref{prop-Hfib-char-ball}.
  Assertions \ref{lem: lrp fundamental ass: open imt}, \ref{lem: lrp fundamental ass: open map}, \ref{lem: lrp fundamental ass: comp} follow immediately from \cref{def: lrp}.

  Next, we prove \cref{lem: lrp fundamental ass: bc}.
  Let \(x'\in X'\) be a point.
  Write \(p: X'\to X\) for the natural projection.
  Since \(f\) is a \morphLHF, there exists a positive real number \(r > 0\) such that for any point \(x_0\in\ball_X(p(x'), r)\) and any \(0 < s_0 < r - d_X(p(x'), x_0)\), it holds that \(f(\ball_X(x_0, s_0)) = \ball_Y(f(x_0), s_0)\).
  Let \(x'_0\in\ball_{X'}(x', r)\) be a point, let \(0 < s'_0 < r - d_{X'}(x', x'_0)\), and let \(y'\in \ball_{Y'}(f'(x'_0), s'_0)\) be a point.
  Since \(p\) is a morphism in \(\Met\), \(0 < s'_0 < r - d_X(p(x'), p(x'_0))\). 
  Moreover, since \(q\) is a morphism in \(\Met\),
  \[q(y')\in \ball_Y(q(f'(x'_0)), s'_0) = \ball_Y(f(p(x'_0)), s'_0) = f(\ball_X(p(x'_0), s'_0)).\]
  Hence, there exists a point \(x\in \ball_X(p(x'_0), s'_0)\) such that \(f(x) = q(y')\).
  Then,
  \[d_{X'}(x'_0, (x, y')) = \max\{d_X(p(x'_0), x), d_{Y'}(f'(x'_0), y')\} < s'_0.\]
  In particular, \((x,y')\in \ball_{X'}(x'_0, s'_0)\).
  This implies that
  \[\ball_{Y'}(f'(x'_0), s'_0)\subset f'(\ball_{X'}(x'_0, s'_0))\subset \ball_{Y'}(f'(x'_0), s'_0).\]
  This completes the proof of \cref{lem: lrp fundamental}.
\end{proof}

Next, we define a Grothendieck topology on \(\Met\).

\begin{definition}\label{defi: cov in Met}
  Let \(X\in \tbMet\) be an object and \(\{f_i:U_i\to X\}_{i\in I}\) a family of morphisms in \(\tbMet\).
  We say that \(\{f_i:U_i\to X\}_{i\in I}\) is an \textbf{\covlhf} of \(X\) if the following conditions hold:
  \begin{condition}
    \item \label{defi: cov in Met enumi: lhf condi: lhf}
    For any \(i\), \(f_i:U_i\to X\) is a \morphLHF.
    \item \label{defi: cov in Met enumi: lhf condi: 3pt}
    For any points \(x_1,x_2, x_3\in X\), and any positive real number \(\ep > 0\), there exist an element \(i\in I\) and points \(u_1\in f_i^{-1}(x_1)\), \(u_2\in f_i^{-1}(x_2)\), \(u_3\in f_i^{-1}(x_3)\) such that \(d_{U_i}(u_1, u_2) < d_X(x_1, x_2) + \ep\), and \(d_{U_i}(u_2, u_3) < d_X(x_2, x_3) + \ep\).
  \end{condition}
\end{definition}

\begin{remark}\label{rmk: cov in Met}
  \ 
  \begin{enumerate}
    \item The term ``lsm'' stands for ``\textbf{l}ocal \textbf{s}ub\textbf{m}etry''.
    \item \label{rmk: cov in Met enumi: surj}
    Let \(X\in \tbMet\) be an object and \(\{f_i:U_i\to X\}_{i\in I}\) a family of morphisms in \(\tbMet\).
    If \(\{f_i:U_i\to X\}_{i\in I}\) is an \covlhf of \(X\), then it holds that \(X = \bigcup_{i\in I}\im(f_i)\).
  \end{enumerate}
\end{remark}


\begin{lemma}\label{lem: top on Met}
  The following assertions hold:
  \begin{assertion}
    \item \label{lem: top on Met ass: lhf cov from H-fib}
    For any \Hfib \(f:X\to Y\) in \(\tbMet\) and any positive real number \(r > 0\), the family of morphisms
    \[\mcU(f,r)\dfn \left\{ B_X(x_1,r)\cup B_X(x_2,r)\cup B_X(x_3,r)\inj X \xto{f} Y\,\middle|\, x_1,x_2,x_3\in X\right\}\]
    is an \covlhf of \(Y\).
    In particular, for any isomorphism \(\iota:X\tosim Y\) in \(\tbMet\), \(\{\iota\}\) is an \covlhf.
    \item \label{lem: top on Met ass: bc}
    Let \(f:Y\to X\) be a morphism in \(\tbMet\) and \(\{f_i:U_i\to X\}_{i\in I}\) an \covlhf.
    Then, the family of morphisms \(\{f_i\times_X \id_Y: U_i\times_X Y\to Y\}_{i\in I}\) is an \covlhf of \(Y\).
    \item \label{lem: top on Met ass: comp}
    Let \(X\in \tbMet\) be an object, \(\{f_i:U_i\to X\}_{i\in I}\) an \covlhf in \(\tbMet\), and \(\{g_{ij}: U_{ij}\to U_i\}_{j\in J_i}\) (\(i\in I\)) \covlhf s.
    Then, \(\{f_i\circ g_{ij}: U_{ij}\to X\}_{j\in J_i, i\in I}\) is an \covlhf.
  \end{assertion}
  In particular, \covlhf s define a Grothendieck pretopology (cf. \cite[Definition 0.31]{Johnstone}, \cite[\href{https://stacks.math.columbia.edu/tag/00V0}{Tag 00V0}, \href{https://stacks.math.columbia.edu/tag/00VH}{Tag 00VH}]{stacks-project}) on \(\tbMet\).
\end{lemma}

\begin{proof}
  First, we prove \cref{lem: top on Met ass: lhf cov from H-fib}.
  Let \(f:X\to Y\) be a \Hfib in \(\tbMet\) and \(r > 0\) a positive real number.
  By \cref{lem: lrp fundamental} \ref{lem: lrp fundamental ass: hf} \ref{lem: lrp fundamental ass: open imt}, for any points \(x_1,x_2,x_3\in X\), the composite \(B_X(x_1,r)\cup B_X(x_2,r)\cup B_X(x_3,r)\subset X \xto{f} Y\) is a \morphLHF.
  Hence, the family \(\mcU(f,r)\) satisfies \cref{defi: cov in Met enumi: lhf condi: lhf} of \cref{defi: cov in Met}.
  Let \(y_1,y_2,y_3\in Y\) be points, \(\ep > 0\) a positive real number, and \(x_2\in X\) a point such that \(y_2 = f(x_2)\).
  By \cref{prop-Hfib-char} ``\ref{prop-Hfib-char-Hfib}\(\Leftrightarrow\)\ref{prop-Hfib-char-ineq}'', there exist points \(x_1, x_3\in X\) such that \(f(x_1)=y_1\), \(f(x_3)=y_3\), \(d_X(x_1, x_2) < d_Y(y_1, y_2) + \min\{\ep,r\}\), and \(d_X(x_2, x_3) < d_Y(y_2, y_3) + \min\{\ep,r\}\).
  This completes the proof of \cref{lem: top on Met ass: lhf cov from H-fib}.
  \Cref{lem: top on Met ass: bc} follows immediately from \cref{lem: lrp fundamental} \ref{lem: lrp fundamental ass: bc} and \cref{defi: cov in Met}.

  Next, we prove \cref{lem: top on Met ass: comp}.
  Let \(\{f_i:U_i\to X\}_{i\in I}\), \(\{g_{ij}: U_{ij}\to U_i\}_{j\in J_i}\) (\(i\in I\)) be lsm coverings.
  By \cref{lem: lrp fundamental} \ref{lem: lrp fundamental ass: comp}, for any index \(i\in I\) and any index \(j\in J_i\), \(f_i\circ g_{ij}\) is a \morphLHF.
  Hence the family \(\{f_i\circ g_{ij}: U_{ij}\to X\}_{j\in J_i, i\in I}\) satisfies \cref{defi: cov in Met enumi: lhf condi: lhf} of \cref{defi: cov in Met}.
  Let \(x_1, x_2, x_3\in X\) be points and \(\ep > 0\) a positive real number.
  Since \(\{f_i:U_i\to X\}_{i\in I}\) is an \covlhf, there exist an index \(i\in I\) and points \(u_*\in f_i^{-1}(x_*)\) (\(*\in \{1,2,3\}\)) such that \(d_{U_i}(u_1, u_2) < d_X(x_1, x_2) + \ep/2\), and \(d_{U_i}(u_2, u_3) < d_X(x_2, x_3) + \ep/2\).
  Since \(\{g_{ij}: U_{ij}\to U_i\}_{j\in J_i}\) is an \covlhf, there exist an index \(j\in J_i\) and points \(u_*'\in g_{ij}^{-1}(u_*)\) (\(*\in\{1,2,3\}\)) such that \(d_{U_{ij}}(u_1', u_2') < d_{U_i}(u_1, u_2) + \ep/2\), and \(d_{U_{ij}}(u_2', u_3') < d_{U_i}(u_2, u_3) + \ep/2\).
  Then, it holds that
  \begin{align*}
    &d_{U_{ij}}(u_1', u_2') < d_{U_i}(u_1, u_2) + \ep/2 < d_X(x_1, x_2) + \ep, \, \text{and} \\
    &d_{U_{ij}}(u_2', u_3') < d_{U_i}(u_2, u_3) + \ep/2 < d_X(x_2, x_3) + \ep.
  \end{align*}
  This completes the proof of \cref{lem: top on Met}.
\end{proof}

\begin{definition}\label{def-lhf-top}
  For any object \(X\in\tbMet\), we write \(J_{\lhftop}(X)\) for the set of sieves on \(X\) that are generated by an lsm covering in \(\tbMet\).
  We say that \(J_{\lhftop}\) is the \textbf{lsm topology} on \(\tbMet\); note that by \cref{lem: top on Met}, the lsm topology is a Grothendieck topology in the sense of \cite[Definition 0.32]{Johnstone}.
  We write \(\tbMet_{\lhftop}\) for the site \((\tbMet, J_{\lhftop})\).
  We say that \(\tbMet_{\lhftop}\) is the \textbf{lsm site}.
  We write \(\Sh_{\lhftop}(\tbMet)\) for the category of sheaves on the site \(\tbMet_{\lhftop}\).
\end{definition}

\begin{lemma}\label{lem: lhf subcan}
  The following assertions hold:
  \begin{assertion}
    \item \label{lem: lhf subcan ass: subcan}
    For any extended pseudometric space \(X\in \epMet\), the functor \(X:\tbMet^{\op}\to \sfSet\) \confer{\cref{notation: PEmet subset PSh} \ref{notation: PEmet subset PSh enumi: full sub}} is a sheaf on the site \(\tbMet_{\lhftop}\). 
    In particular, the site \(\tbMet_{\lhftop}\) is subcanonical (that is, every representable functor is a sheaf; cf. \cite[\S 0.3]{Johnstone}).
    \item \label{lem: lhf subcan ass: hom sh}
    For any extended pseudometric spaces \(X,Y\in \epMet\), the functor 
    \begin{align*}
      \inHom_{\epMet}(X, Y): \tbMet^{\op} &\to \sfSet, \\
      Z &\mapsto \Hom_{\epMet_{/Z}}(X\times Z, Y\times Z)
    \end{align*}
    is a sheaf on the site \(\tbMet_{\lhftop}\). 
    \item \label{lem: lhf subcan ass: isom sh}
    For any extended pseudometric spaces \(X,Y\in \epMet\), the functor 
    \begin{align*}
      \inIsom_{\epMet}(X, Y): \tbMet^{\op} &\to \sfSet, \\
      Z &\mapsto \Isom_{\epMet_{/Z}}(X\times Z, Y\times Z)
    \end{align*}
    is a sheaf on the site \(\tbMet_{\lhftop}\). 
  \end{assertion}
\end{lemma}

\begin{proof}
  First, we prove \cref{lem: lhf subcan ass: subcan}.
  Let \(X\in \epMet\) be an object, \(T\in \tbMet\) an object, and \(\{f_i:U_i\to T\}_{i\in I}\) an lsm covering of \(T\).
  Write
  \[
    E \dfn \mathrm{Equalizer}\left(\prod_{i\in I}\Hom_{\epMet}(U_i, X) \tto \prod_{i_0,i_1\in I}\Hom_{\epMet}(U_{i_0}\times U_{i_1}, X)\right).
  \]
  Since \(T = \bigcup_{i\in I}\im(f_i)\), the natural map \(\Hom_{\epMet}(T, X)\to E\) is injective.
  Let \((g_i)_{i\in I}\in E\) be an element and \(t\in T\) a point.
  Then, there exist an index \(i\in I\) and a point \(u_i\in U_i\) such that \(f_i(u_i)=t\).
  Write \(g(t)\dfn g_i(u_i)\).
  Since \((g_i)_{i\in I}\in E\), the map of sets \(g: \uSet{T}\to \uSet{X}\) is well-defined and satisfies that \(g\circ f_i = g_i\) for each \(i\in I\).
  Let \(t_0, t_1\in T\) be points and \(\ep > 0\) a positive real number.
  By \cref{defi: cov in Met} \ref{defi: cov in Met enumi: lhf condi: 3pt}, there exist an index \(i'\in I\) and points \(u_0, u_1\in U_{i'}\) such that \(f_{i'}(u_0)=t_0\), \(f_{i'}(u_1)=t_1\), and \(d_{U_{i'}}(u_0,u_1) < d_T(t_0,t_1) + \ep\).
  Hence, it holds that
  \[
    d_X(g(t_0), g(t_1)) = d_X(g_{i'}(u_0), g_{i'}(u_1)) \leq d_{U_{i'}}(u_0,u_1) < d_T(t_0,t_1) + \ep.
  \]
  By letting \(\ep\to 0\), we conclude that \(g\) is a morphism in \(\epMet\).
  This completes the proof of \cref{lem: lhf subcan}.

  Next, we prove \cref{lem: lhf subcan ass: hom sh}. 
  Let \(\left\{p_i:U_i\to B\right\}_{i\in I}\) be an lsm covering of \(B\).
  Since \(\bigcup_{i\in I}\im(p_i) = B\), the natural map \(\Hom_{\epMet_{/B}}(X\times B, Y\times B)\to \prod_{i\in I}\Hom_{\epMet_{/U_i}}(X\times U_i, Y\times U_i)\) is injective.
  Let \((x,b)\in X\times B\) be a point and
  \[(f_i)_{i\in I}\in \prod_{i\in I}\Hom_{\epMet_{/U_i}}(X\times U_i, Y\times U_i)\]
  an element such that for each pair of indices \((i,j)\in I^2\), it holds that \(f_i\times_{U_i}\id_{U_i\times_{B}U_j} = f_j\times_{U_j}\id_{U_i\times_{B}U_j}\).
  Since \(\left\{p_i:U_i\to B\right\}_{i\in I}\) is an lsm covering, there exist an index \(i\in I\) and a point \(u_i\in U_i\) such that \(p_i(u_i)=b\).
  Then, the pair \((x,u_i)\) is a point of \(X\times U_i\).
  We define \[f(x,b)\dfn (\id_Y\times p_i)(f_i(x,u_i)).\]
  For any index \(j\in I\) and any point \(u_j\in U_j\) such that \(p_j(u_j)=b\), since \(f_i\times_{U_i}\id_{U_i\times_{B}U_j} = f_j\times_{U_j}\id_{U_i\times_{B}U_j}\), it holds that \((f_i(x,u_i), u_i, u_j) = (f_j(x,u_j),u_i,u_j)\).
  This implies that \((\id_Y\times p_i)(f_i(x,u_i)) = (\id_Y\times p_j)(f_j(x,u_j))\).
  In particular, the map \(f:X\times B\to Y\times B\) is well-defined.

  Let \((x_1,b_1), (x_2,b_2)\in X\times B\) be points and \(\ep > 0\) a positive real number.
  Since \(\left\{p_i:U_i\to B\right\}_{i\in I}\) is an lsm covering, there exist an index \(i\in I\) and points \(u_1,u_2\in U_i\) such that \(b_1=p_i(u_1)\), \(b_2=p_i(u_2)\), and \(d_{U_i}(u_1,u_2) < d_B(b_1,b_2) + \ep\).
  Then, it holds that
  \begin{align*}
    d_{Y\times B}(f(x_1, b_1), f(x_2, b_2))
    &\leq d_{Y\times U_i}(f_i(x_1, u_1), f_i(x_2, u_2)) \\
    &\leq d_{X\times U_i}((x_1, u_1), (x_2, u_2)) \\
    &= d_{(X\times B)\times_B U_i}((x_1, b_1, u_1), (x_2, b_2, u_2)) \\
    &= \max\left\{d_{X\times B}((x_1,b_1), (x_2, b_2)), d_{U_i}(u_1,u_2)\right\} \\
    &< d_{X\times B}((x_1,b_1), (x_2, b_2)) + \ep.
  \end{align*}
  By letting \(\ep\to 0\), we conclude that \(d_{Y\times B}(f(x_1, b_1), f(x_2, b_2)) \leq d_{X\times B}((x_1,b_1), (x_2, b_2))\).
  This implies that \(f\) is a morphism in \(\epMet\).
  By the definition of \(f\), one may verify immediately that for any \(i\in I\), \(f_i = f\times \id_{U_i}\).
  This completes the proof of \cref{lem: lhf subcan ass: hom sh}. 

  Finally, \cref{lem: lhf subcan ass: isom sh} follows immediately from \cref{lem: lhf subcan ass: hom sh}. 
  This completes the proof of \cref{lem: lhf subcan}. 
\end{proof}

\begin{remark}
  Conversely, if a Grothendieck topology \(J\) on \(\tbMet\) is subcanonical, then one can easily prove that for any object \(X\in\tbMet\), any points \(x,x'\in X\), any covering sieve \(\mcS\in J(X)\), and any positive real number \(\ep > 0\), there exist an object \(Y\in \tbMet\), a morphism \(f:Y\to X\) in \(\mcS(Y)\), and points \(y,y'\in Y\) such that \(f(y)=x\), \(f(y')=x'\), and \(d_Y(y,y') < d_X(x,x') + \ep\). 
\end{remark}

Next, we verify that any descent datum with respect to the lsm topology is effective.

\begin{lemma}\label{lem: eff descent}
  Let \(X\in \tbMet\) be an object,
  \(\{f_i:U_i\to X\}_{i\in I}\) an lsm covering of \(X\),
  \((p_i:P_i\to U_i)_{i\in I}\) a family of morphisms in \(\epMet\), and
  \[
    (\varphi_{ij}:P_i\times_{p_i, U_i}(U_i\times_XU_j)\tosim P_j\times_{p_j, U_j}(U_i\times_XU_j))_{(i,j)\in I^2}
  \]
  a family of isomorphisms in \(\epMet\).
  Assume that the pair \(((p_i)_{i\in I}, (\varphi_{ij})_{(i,j)\in I^2})\) satisfies the following conditions:
  \begin{condition}
    \item \label{lem: eff descent condi: 1}
    For any pair of indices \((i,j)\in I^2\), \(\varphi_{ij}\) is an isomorphism in \(\epMet_{/U_i\times_XU_j}\).
    \item \label{lem: eff descent condi: cocycle}
    For any triple of indices \((i,j,k)\in I^3\),
    \[(\varphi_{jk}\times_{U_j\times_XU_k}\id_{U_i\times_XU_j\times_XU_k})\circ (\varphi_{ij}\times_{U_i\times_XU_j}\id_{U_i\times_XU_j\times_XU_k}) = \varphi_{ik}\times_{U_i\times_XU_k}\id_{U_i\times_XU_j\times_XU_k}.\]
  \end{condition}
  Then, the following assertions hold:
  \begin{assertion}
    \item \label{lem: eff descent ass: eff}
    There exist a morphism \(p:P\to X\) in \(\epMet\) and a family of isomorphisms \[(\varphi_i:P_i\tosim P\times_{p,X}U_i)_{i\in I}\] in \(\epMet\) such that for each index \(i\in I\), \(\varphi_i\) is an isomorphism in \(\epMet_{/U_i}\), and, moreover, for any pair of indices \((i,j)\in I^2\), \((\varphi_i\times_{U_i}\id_{U_i\times_XU_j}) \circ\varphi_{ij} = \varphi_j\times_{U_j}\id_{U_i\times_XU_j}\).
    \item \label{lem: eff descent ass: met}
    If \((p_i:P_i\to U_i)_{i\in I}\) is a family of morphisms in \(\eMet\) (resp. in \(\pMet\), in \(\Met\)), then \(P\in\eMet\) (resp. \(P\in\pMet\), \(P\in\Met\)). 
    \item \label{lem: eff descent ass: colim}
    \(P\) is isomorphic as an object of \(\epMet_{/X}\) to the colimit of the diagram in \(\epMet_{/X}\) determined by the pair \(((p_i)_{i\in I}, (\varphi_{ij})_{(i,j)\in I^2})\).
  \end{assertion}
\end{lemma}

\begin{proof}
  First, we prove \cref{lem: eff descent ass: eff}.
  For any indices \(i,j\in I\) and any points \(x_i\in P_i\), \(x_j\in P_j\), we define
  \[
    (i, x_i)\sim (j, x_j) \, \deff\,
    \left\{\begin{aligned}
      &\bullet \ f_i(p_i(x_i)) = f_j(p_j(x_j)), \ \ \text{and} \\
      &\bullet \ \varphi_{ij}(x_i, p_i(x_i), p_j(x_j)) = (x_j, p_i(x_i), p_j(x_j)).
    \end{aligned}\right.
  \]
  Then, by \cref{lem: eff descent condi: cocycle}, the relation \(\sim\) is an equivalence relation on the set \(\coprod_{i\in I}\uSet{P_i}\).
  We define \(\uSet{P}\dfn \left(\coprod_{i\in I}\uSet{P_i}\right)/\sim\).
  Write \(p:\uSet{P}\to \uSet{X}\) for the map of sets induced by \(\coprod_{i\in I}f_i\circ p_i:\coprod_{i\in I}\uSet{P_i}\to \uSet{X}\).
  For any \(i\in I\), write \(g_i:\uSet{P_i}\to \uSet{P}\) for the natural map.
  Then, by the definition of \(\uSet{P}\) and the equivalence relation \(\sim\), for any index \(i\in I\) and any point \(u\in U_i\), \(g_i|_{p_i^{-1}(u)}:p_i^{-1}(u)\to \uSet{P}\) induces a bijection \(p_i^{-1}(u)\tosim p^{-1}(f_i(u))\).
  This implies that for any index \(i\in I\), the following diagram of sets is Cartesian:
  \begin{equation}
    \label{lem: eff descent eq: diag}
    \begin{tikzcd}
      {\uSet{P_i}} \ar[r, "{g_i}"] \ar[d, "{p_i}"] & {\uSet{P}}\ar[d, "{p}"] \\
      {\uSet{U_i}} \ar[r, "{f_i}"] & {\uSet{X}}.
    \end{tikzcd}
    \tag{\(\dagger\)}
  \end{equation}
  
  Next, we prove the following claim:
  \begin{claim}\label{lem: eff descent claim: compare}
    Let \(i,j\in I\) be indices; \(x_0^i,x_1^i\in P_i\) and \(x_0^j,x_1^j\in P_j\) points.
    Assume that \((i,x_0^i)\sim (j, x_0^j)\) and that \((i,x_1^i)\sim (j, x_1^j)\).
    Then, it holds that
    \[d_{P_i}(x_0^i,x_1^i) \leq \max\{d_{P_j}(x_0^j, x_1^j), d_{U_i}(p_i(x_0^i), p_i(x_1^i))\}.\]
  \end{claim}
  \noindent
  By the definition of the relation \(\sim\), for each \(\bullet\in\{0,1\}\), it holds that \(\varphi_{ij}(x_{\bullet}^i, p_i(x_{\bullet}^i), p_j(x_{\bullet}^j)) = (x_{\bullet}^j, p_i(x_{\bullet}^i), p_j(x_{\bullet}^j))\).
  Hence,
  \begin{align*}
    d_{P_i}(x_0^i,x_1^i)
    &\leq d_{P_i\times_{p_i,U_i}(U_i\times_XU_j)}((x_0^i, p_i(x_0^i), p_j(x_0^j)), (x_1^i, p_i(x_1^i), p_j(x_1^j))) \\
    &= d_{P_j\times_{p_j,U_j}(U_i\times_XU_j)}((x_0^j, p_i(x_0^i), p_j(x_0^j)), (x_1^j, p_i(x_1^i), p_j(x_1^j))) \\
    &= \max\{d_{P_j}(x_0^j, x_1^j), d_{U_i}(p_i(x_0^i), p_i(x_1^i)), d_{U_j}(p_j(x_0^j), p_j(x_1^j))\} \\
    &= \max\{d_{P_j}(x_0^j, x_1^j), d_{U_i}(p_i(x_0^i), p_i(x_1^i))\}.
  \end{align*}
  This completes the proof of \cref{lem: eff descent claim: compare}.

  Let us continue to prove \cref{lem: eff descent}.
  For any points \(x_0,x_1\in \uSet{P}\), we define
  \[
    d_P(x_0,x_1)\dfn \inf\left\{d_{P_i}(a_0,a_1)\,\middle|\,i\in I, a_0,a_1\in P_i,\, g_i(a_0) = x_0,\, g_i(a_1) = x_1\right\}.
  \]
  Then, it holds that \(d_P(x_0, x_1) = d_P(x_1, x_0)\) for any points \(x_0,x_1\in \uSet{P}\).
  To prove that \((\uSet{P}, d_P)\in \epMet\), it suffices to prove that \(d_P\) satisfies the triangle inequality.
  Let \(x_0,x_1,x_2\in \uSet{P}\) be points and \(\ep > 0\) a positive real number.
  Since the diagram \cref{lem: eff descent eq: diag} is Cartesian, it follows from \cref{defi: cov in Met} \ref{defi: cov in Met enumi: lhf condi: 3pt} that there exist an index \(i\in I\) and points \(x_0^i\in g_i^{-1}(x_0)\), \(x_1^i\in g_i^{-1}(x_1)\), \(x_2^i\in g_i^{-1}(x_2)\) such that \(d_{U_i}(p_i(x_0^i), p_i(x_1^i)) < d_X(p(x_0), p(x_1)) + \ep\), and \(d_{U_i}(p_i(x_1^i), p_i(x_2^i)) < d_X(p(x_1), p(x_2)) + \ep\).
  Then, by \cref{lem: eff descent claim: compare}, for all indices \(\lozenge, \blacklozenge\in \{0,1,2\}\), and index \(j\in I\), and any points \(x_{\lozenge}^j, x_{\blacklozenge}^j\in P_j\) such that \((i,x_{\lozenge}^i) \sim (j,x_{\lozenge}^j)\) and \((i,x_{\blacklozenge}^i) \sim (j,x_{\blacklozenge}^j)\), it holds that
  \begin{align*}
    d_{P_i}(x_{\lozenge}^i, x_{\blacklozenge}^i)
    &\leq \max\left\{ d_{P_j}(x_{\lozenge}^j, x_{\blacklozenge}^j), d_{U_i}(p_i(x_{\lozenge}^i), p_i(x_{\blacklozenge}^i))\right\} \\
    &\leq \max\left\{ d_{P_j}(x_{\lozenge}^j, x_{\blacklozenge}^j), d_X(p(x_{\lozenge}), p(x_{\blacklozenge})) + \ep \right\} \\
    &\leq \max\left\{ d_{P_j}(x_{\lozenge}^j, x_{\blacklozenge}^j), d_X(p(x_{\lozenge}), p(x_{\blacklozenge}))\right\} + \ep  \\
    &\leq d_{P_j}(x_{\lozenge}^j, x_{\blacklozenge}^j) + \ep.
  \end{align*}
  In particular, for all indices \(\lozenge, \blacklozenge\in \{0,1,2\}\), \(d_{P_i}(x_{\lozenge}^i, x_{\blacklozenge}^i) \leq d_P(x_{\lozenge}, x_{\blacklozenge}) + \ep\).
  Thus,
  \[
    d_P(x_0,x_2) \leq d_{P_i}(x_0^i, x_2^i) \leq d_{P_i}(x_0^i, x_1^i) + d_{P_i}(x_1^i, x_2^i) \leq d_P(x_0,x_1) + d_P(x_1,x_2) + 2\ep.
  \]
  By letting \(\ep \to 0\), we conclude that \(d_P\) satisfies the triangle inequality.
  Write \(P\dfn (\uSet{P},d_P)\).

  Let \(x_0,x_1\in P\) be points.
  Then, it holds that
  \begin{align*}
    d_P(x_0,x_1) &= \inf\left\{d_{P_i}(a_0,a_1)\,\middle|\,i\in I, \ a_0\in g_i^{-1}(x_0), \ a_1\in g_i^{-1}(x_1) \right\} \\
    &\geq \inf\left\{d_{U_i}(p_i(a_0),p_i(a_1))\,\middle|\,i\in I, \ a_0\in g_i^{-1}(x_0), \ a_1\in g_i^{-1}(x_1) \right\} \\
    &\geq \inf\left\{d_X(f_i(p_i(a_0)),f_i(p_i(a_1)))\,\middle|\,i\in I, \ a_0\in g_i^{-1}(x_0), \ a_1\in g_i^{-1}(x_1) \right\} \\
    &= d_X(p(x_0), p(x_1)).
  \end{align*}
  This implies that \(p: P\to X\) is a morphism in \(\epMet\).

  By the definition of the distance \(d_P\), for any index \(i\in I\), \(g_i:P_i\to P\) is a morphism in \(\epMet\).
  This implies that for any index \(i\in I\), the natural bijection \(\varphi_i:P_i\to P\times_{p,X}U_i\) induced by the diagram \cref{lem: eff descent eq: diag} is a morphism in \(\epMet\).
  For any index \(i\in I\) and any points \(x_0,x_1\in P_i\), it follows from \cref{lem: eff descent claim: compare} that
  \[d_{P_i}(x_0,x_1) = \max\{d_P(g_i(x_0),g_i(x_1)), d_{U_i}(p_i(x_0),p_i(x_1))\}.\]
  Thus, for any index \(i\in I\), \(\varphi_i:P_i\to P\times_{p,X}U_i\) is an isomorphism in \(\epMet\).
  Furthermore, for any indices \(i,j\in I\), it follows from the definition of the equivalence relation \(\sim\) that \((\varphi_i\times_{U_i}\id_{U_i\times_XU_j})\circ \varphi_{ij} = \varphi_j\times_{U_j}\id_{U_i\times_XU_j}\).
  This completes the proof of \cref{lem: eff descent ass: eff}.

  Next, we prove \cref{lem: eff descent ass: met}.
  If \((p_i:P_i\to U_i)_{i\in I}\) is a family of morphisms in \(\pMet\), then by \cref{defi: cov in Met} \ref{defi: cov in Met enumi: lhf condi: 3pt}, for any points \(x_0,x_1\in P\), \(d_P(x_0,x_1) = d_P(x_0,x_1) < \infty\). 
  Hence, \(P\in \pMet\).
  Assume that \((p_i:P_i\to U_i)_{i\in I}\) is a family of morphisms in \(\eMet\).
  Let \(x_0,x_1\in P\) be points such that \(d_P(x_0, x_1)=0\).
  Since \(X\in\tbMet\), \(p(x_0) = p(x_1)\). 
  Since the diagram \eqref{lem: eff descent eq: diag} is Cartesian, there exist an index \(i\in I\) and points \(x_0^i, x_1^i\in P_i\) such that \(g_i(x_0^i) = x_0\), \(g_i(x_1^i) = x_1\), and \(p_i(x_0^i) = p_i(x_1^i)\).
  Then, \(d_{P_i}(x_0^i, x_1^i)=0\). 
  Since \(P_i\in\eMet\), \(x_0^i=x_1^i\).
  Thus, \(x_0=g_i(x_0^i)=g_i(x_1^i)=x_1\). 
  In particular, \(P\in\eMet\). 
  This completes the proof of \cref{lem: eff descent ass: met}.

  Finally, \cref{lem: eff descent ass: colim} follows immediately from \cref{lem: lim and colim in PEMet} \ref{lem: lim and colim in PEMet ass: colim}, \cref{defi: cov in Met} \ref{defi: cov in Met enumi: lhf condi: 3pt}, and \cref{lem: eff descent ass: eff}.
  This completes the proof of \cref{lem: eff descent}.
\end{proof}

\begin{definition}\label{def-stack-all-met}
  We write \(\stAllMet\) for the category fibered in groupoids (cf.~\cite[\href{https://stacks.math.columbia.edu/tag/003T}{Tag 003T}]{stacks-project}) over \(\tbMet\), defined as follows:
  \begin{itemize}
    \item 
    An object of \(\stAllMet\) is a morphism \(p:X\to T\) in \(\Met\) such that \(T\in \tbMet\).
    \item 
    A morphism \((f_d, f_c)\) from \(p_1:X_1\to T_1\) to \(p_2:X_2\to T_2\) in \(\stAllMet\) is a commutative diagram 
    \[\begin{tikzcd}
      {X_1} \ar[r, "{f_d}"] \ar[d, "{p_1}"] & {X_2} \ar[d, "{p_2}"] \\
      {T_1} \ar[r, "{f_c}"] & {T_2}
    \end{tikzcd}\]
    in \(\Met\) such that the induced morphism \(X_1\to X_2\times_{T_2}T_1\) is an isomorphism.
    \item 
    Composition in \(\stAllMet\) is defined in the evident way from composition in \(\Met\).
  \end{itemize}
  We define 
  \begin{align*}
    \cod:\stAllMet &\to \tbMet, \\
    (p:X\to T) &\mapsto T, \\
    (f_d, f_c) &\mapsto f_c.
  \end{align*}
\end{definition}

\begin{corollary}\label{cor: Mor is a stack}
  \(\cod:\stAllMet\to \tbMet, [P\to X]\mapsto X\) is a stack in groupoids (cf. \cite[\href{https://stacks.math.columbia.edu/tag/02ZI}{Tag 02ZI}]{stacks-project}) over \(\tbMet_{\lhftop}\).
\end{corollary}

\begin{proof}
  \Cref{cor: Mor is a stack} follows immediately from \cref{lem: lhf subcan} \ref{lem: lhf subcan ass: isom sh} and \cref{lem: eff descent} \ref{lem: eff descent ass: eff} \ref{lem: eff descent ass: met}.
\end{proof}

\begin{remark}
  Conversely, if a Grothendieck topology \(J\) on \(\tbMet\) makes the category fibered in groupoids \(\stAllMet\to \tbMet\) a stack on \(\tbMet\) with respect to \(J\), then one can easily prove that for any object \(X\in\tbMet\), any points \(x_1,x_2,x_3\in X\), any covering sieve \(\mcS\in J(X)\), and any positive real number \(\ep > 0\), there exist an object \(Y\in \tbMet\), a morphism \(f:Y\to X\) in \(\mcS(Y)\), and points \(y_1,y_2,y_3\in Y\) such that \(f(y_i)=x_i\) (\(\forall i\in\{1,2,3\}\)), \(d_Y(y_1,y_2) < d_X(x_1,x_2) + \ep\), and \(d_Y(y_2,y_3) < d_X(x_2,x_3) + \ep\). 
\end{remark}

\section{Naive Metric Stack}
\label{section: naive metric stack}

In this section, we develop the basic theory of stacks over the site \(\tbMet_{\lhftop}\). 
We first introduce \textbf{naive metric stacks} as a class of stacks on \(\tbMet_{\lhftop}\) and establish their fundamental properties.
We then define good moduli spaces as morphisms to extended pseudometric spaces satisfying an appropriate universality property, and investigate their structural features.
This notion is the metric-space-theoretic analogue of good moduli spaces for algebraic stacks in algebraic geometry \confer{\cite{Alper-good-moduli}}; 
however, in contrast to the algebraic setting, every stack in the metric context admits a good moduli space \confer{\cref{prop-epmet-into-stacks-has-left-adj} \ref{prop-epmet-into-stacks-has-left-adj-ep}}.
Finally, we study conditions under which a good moduli space is separated.

\begin{notation}\label{notation-stacks}
  \ 
  \begin{enumerate}
    \item 
    We write \(\Stack_{\lhftop}(\tbMet)\) for the 2-category of stacks in groupoids over \(\tbMet_{\lhftop}\) and \(\Sh_{\lhftop}(\tbMet)\) for the category of sheaves on \(\tbMet_{\lhftop}\).
    For any stack in groupoids \(\mcX\) over \(\tbMet_{\lhftop}\) and any object \(T\in \epMet\), by a slight abuse of notation, we write \(\mcX(T)\dfn \Hom_{\Stack_{\lhftop}(\tbMet)}(T,\mcX)\). 
    By the 2-Yoneda lemma \confer{\cite[\href{https://stacks.math.columbia.edu/tag/0GWI}{Tag 0GWI}]{stacks-project}}, we often identify \(\mcX(T)\) with the fiber of \(\mcX\to \tbMet_{\lhftop}\) at \(T\in\tbMet_{\lhftop}\).
    Note that \(\mcX(T)\) is a groupoid.
    For any morphism \(f:\mcX\to \mcY\) of stacks in groupoids over \(\tbMet_{\lhftop}\) and any object \(T\in \tbMet\), we write \(f_T:\mcX(T)\to \mcY(T)\) for the induced morphism of groupoids.
    We write \(\pi_0:\Stack_{\lhftop}(\tbMet)\to \Sh_{\lhftop}(\tbMet)\) for the functor that assigns to each stack its sheaf of path connected components; 
    that is, for any object \(\mcX\in\Stack_{\lhftop}(\tbMet)\), \(\pi_0(\mcX)\) is the sheafification of the presheaf sending an object \(U\in\tbMet\) to the set of isomorphism classes of the groupoid \(\mcX(U)\). 
    \item 
    By a slight abuse of notation, for any sheaf \(\mcF\) on the site \(\tbMet_{\lhftop}\), we also write \(\mcF\) for the stack in groupoids over \(\tbMet_{\lhftop}\) obtained by applying the Grothendieck construction \confer{\cite[\href{https://stacks.math.columbia.edu/tag/02Y2}{Tag 02Y2}]{stacks-project}} to \(\mcF\). 
    By a further abuse of notation, we regard \(\Sh_{\lhftop}(\tbMet)\subset \Stack_{\lhftop}(\tbMet)\) as a full subcategory. 
    \item 
    Let \(\mcX, \mcY, \mcZ\) be stacks in groupoids over \(\tbMet_{\lhftop}\); \(f,f':\mcX\tto \mcY\) and \(g,g':\mcY\tto \mcZ\) morphisms of stacks in groupoids over \(\tbMet_{\lhftop}\); \(\alpha:f\to f'\) and \(\beta:g\to g'\) natural transformations.
    We write \(\beta\star \alpha\) for the horizontal composition of \(\alpha\) and \(\beta\). 
    By a slight abuse of notation, we write \(g\star \alpha\dfn \id_g\star \alpha\) and \(\beta\star f\dfn \beta \star \id_f\). 
    \item \label{notation-stacks-2-prod}
    Let \(f:\mcX\to \mcY \gets \mcZ: g\) be morphisms of stacks in groupoids over \(\tbMet_{\lhftop}\). 
    By a slight abuse of notation, we write \(\mcX\times_{\mcY}\mcZ\) for the 2-fiber product \confer{\cite[\href{https://stacks.math.columbia.edu/tag/003O}{Tag 003O}]{stacks-project}} of the diagram \(f:\mcX\to \mcY \gets \mcZ: g\).
  \end{enumerate}
\end{notation}

\begin{lemma}\label{prop-epmet-into-stacks-has-left-adj}
  \ 
  \begin{assertion}
    \item \label{prop-epmet-into-stacks-has-left-adj-sh}
    The functor \(\pi_0:\Stack_{\lhftop}(\tbMet)\to \Sh_{\lhftop}(\tbMet)\) is a left adjoint to the inclusion functor \(\Sh_{\lhftop}(\tbMet)\subset \Stack_{\lhftop}(\tbMet)\). 
    \item \label{prop-epmet-into-stacks-has-left-adj-ep}
    The functor \(\mu\circ \pi_0:\Stack_{\lhftop}(\tbMet)\to \epMet\) is a left adjoint to the inclusion functor \(\epMet \subset \Sh_{\lhftop}(\tbMet)\subset \Stack_{\lhftop}(\tbMet)\). 
  \end{assertion}
\end{lemma}

\begin{proof}
  \Cref{prop-epmet-into-stacks-has-left-adj-sh} follows immediately from the well-known fact that \(\pi_0:\Grpd\to \sfSet\) is a left adjoint functor to the inclusion functor \(\sfSet\subset \Grpd\).
  \Cref{prop-epmet-into-stacks-has-left-adj-ep} follows immediately from \cref{prop-epmet-into-stacks-has-left-adj-sh} and \cref{lem: univ adj PEMet}.
\end{proof}

\begin{definition}[{Coarse Moduli Space}]\label{def-good-moduli}
  Let \(\mcX\in \Stack_{\lhftop}(\tbMet)\) be a stack in groupoids over \(\tbMet_{\lhftop}\). 
  We write \(\mu(\mcX)\dfn \mu(\pi_0(\mcX))\in \epMet\) and \(\pi_{\mcX}:\mcX\to \mu(\mcX)\) for the natural morphism \confer{\cref{prop-epmet-into-stacks-has-left-adj} \ref{prop-epmet-into-stacks-has-left-adj-ep}}.
  We say that \(\pi_{\mcX}:\mcX\to \mu(\mcX)\) is a \textbf{good moduli space} of \(\mcX\). 
  We say that the good moduli space \(\pi_{\mcX}:\mcX\to \mu(\mcX)\) is a \textbf{coarse moduli space} if \(\mu(\mcX)\in\Met\).
\end{definition}

\begin{remark}
  The term ``good moduli space'' in \cref{def-good-moduli} is borrowed from \cite[Definition 4.1]{Alper-good-moduli}. 
  For an algebraic stack, if a good moduli exists, then by \cite[Theorem 6.6]{Alper-good-moduli}, it is, in many common situations, universal with respect to morphisms to algebraic spaces. 
  On the other hand, for algebraic stacks, a good moduli space need not exist, and even when it does, it does not imply that the connected components sheaf \(\pi_0\) itself is an algebraic space. 
  In contrast, since the embedding functor from the category of extended pseudometric spaces into the category of stacks admits a left adjoint \confer{\cref{prop-epmet-into-stacks-has-left-adj}}, in our situation, every stack in groupoids over \(\tbMet_{\lhftop}\) admits a good moduli space. 
\end{remark}

\begin{lemma}\label{prop-good-moduli-structure}
  Let \(\mcX\in\Stack_{\lhftop}(\tbMet)\) be a stack in groupoids over \(\tbMet\).
  Then, the following assertions hold:
  \begin{assertion}
    \item \label{prop-good-moduli-structure-uset}
    The map of sets \(u:\mcX(*)/\!\simeq \,\to \Hom_{\epMet}(*, \mu(\mcX))\tosim \uSet{\mu(\mcX)}\) is bijective.
    \item \label{prop-good-moduli-structure-dist}
    For any elements \(x,x'\in \mu(\mcX)\) and any positive real number \(\ep > 0\), there exist an integer \(N\in \N\), elements \(r_i\in \bRp\) (\(i\in\{0,\cdots,N\}\)), and morphisms \(f_i:2_{r_i}\to \mcX\) (\(i\in\{0,\cdots,N\}\)) such that the following conditions hold:
    \begin{condition}
      \item 
      \(\sum_{i=0}^N r_i \leq d_{\mu(\mcX)}(x,x') + \ep\). 
      \item 
      It holds that \(u(\iota_0^*f_0) = x\), and \(u(\iota_{r_N}^*f_N) = x'\), where we regard \(f_0\) (resp. \(f_N\)) as an object of \(\mcX(2_{r_0})\) (resp. \(\mcX(2_{r_N})\)) via the 2-Yoneda lemma \confer{\cite[\href{https://stacks.math.columbia.edu/tag/0GWI}{Tag 0GWI}]{stacks-project}}.
      \item 
      For any \(i\in\{1,\cdots, N\}\), it holds that \(\iota_{r_{i-1}}^*f_{i-1} \cong \iota_0^*f_i\) in \(\mcX(*)\). 
    \end{condition}
  \end{assertion}
\end{lemma}

\begin{proof}
  Both assertions of \cref{prop-good-moduli-structure} follow immediately from \cref{lem: univ adj PEMet} \ref{lem: univ adj PEMet ass: uset}, \cref{lem: lhf subcan} \ref{lem: lhf subcan ass: subcan}, and \cref{prop-epmet-into-stacks-has-left-adj} \ref{prop-epmet-into-stacks-has-left-adj-ep}.
\end{proof}

Next, in order to define ``stacks that are covered by metric spaces,'' we extend the notion of submetry to morphisms of stacks in groupoids over \(\tbMet_{\lhftop}\) and introduce a class of morphisms of stacks in groupoids over \(\tbMet_{\lhftop}\) called \textbf{\fHfib}. 
Here we recall the characterization given in \cref{prop-Hfib-char}. 
Let us translate it into the 2-categorical setting.

\begin{definition}\label{def-Hfib-for-stacks}
  Let \(f:\mcX\to \mcY\) be a morphism of stacks in groupoids over \(\tbMet_{\lhftop}\). 
  We say that \(f\) is a \textbf{\fHfib} if the following conditions hold:
  \begin{condition}
    \item \label{def-Hfib-for-stacks-ess-surj}
    The morphism of groupoids \(f_*:\mcX(*)\to \mcY(*)\) is essentially surjective.
    \item 
    for any positive real numbers \(r > s > 0\), any morphisms \(x:*\to \mcX\) and \(g:2_s\to \mcY\) of stacks in groupoids over \(\tbMet_{\lhftop}\), and any 2-morphism \(\alpha:f\circ x \tosim g\circ \iota_{0s}\), there exist a morphism \(h:2_r\to \mcX\) of stacks in groupoids over \(\tbMet_{\lhftop}\) and 2-morphisms \(\beta:h\circ \iota_{0r}\tosim x\), \(\gamma:f\circ h \tosim g\circ \iota_{rs}\) such that \(\alpha\circ (f\star \beta) = \gamma\star \iota_{0r}\): 
    \begin{equation*}
      \begin{tikzcd}
        {*} \ar[rr, "x"] \ar[d, "{\iota_{0r}}"'] 
        &[0.3cm] {\relax} &[0.3cm] 
        {\mcX} \ar[d, "f"]
        \ar[dl, shorten=2mm, Rightarrow, "{\exists \gamma}"{below, xshift=-0.4cm, yshift=0.15cm}, end anchor={[xshift=0.5cm]}] \\
        {2_r} \ar[r, "{\iota_{rs}}"'] \ar[urr, "{\exists h}"] 
        \ar[ur, shorten=3mm, Rightarrow, "{\exists \beta}", end anchor={[xshift=-0.5cm]}]
        & {2_s} \ar[r, "g"'] & {\mcY}
      \end{tikzcd}
    \end{equation*}
  \end{condition}
\end{definition}

\begin{remark}\label{rmk-fHfib}
  The above definition abstracts only the categorical properties of submetries.
  If one aims to extend the theory of metric spaces to stacks in groupoids over \(\tbMet_{\lhftop}\), a morphism of stacks that should genuinely be called a submetry would presumably be required to satisfy additional conditions, beyond those listed above, that more fully reflect the geometric nature of submetries.
  However, determining what such conditions should be lies outside the scope of the present paper, and we therefore do not pursue this issue here.
  For this reason, we refer to the above notion as a \textbf{formal} submetry.
\end{remark}

\begin{definition}\label{def-naive-metric-stack}
  We say that a stack in groupoids \(\mcX\) over \(\tbMet_{\lhftop}\) is a \textbf{naive metric stack} if there exists a metric space \(X\) and a \fHfib of stacks in groupoids \(X\to \mcX\) over \(\tbMet_{\lhftop}\). 
\end{definition}

\begin{remark}\label{rmk-why-naive}
  The reason we attach the term ``naive'' in \cref{def-naive-metric-stack} is as follows.
  \begin{enumerate}
    \item 
    As already pointed out in \cref{rmk-fHfib}, the notion of submetry should be extended to morphisms of stacks over \(\tbMet_{\lhftop}\) in a more appropriate manner.
    In the present work, however, in light of our purposes, we impose only minimal categorical requirements on formal submetries, and for this reason we regard \cref{def-naive-metric-stack} as being merely ``naive.''
    \item \label{rmk-why-naive-BG}
    In ordinary metric geometry, typical examples of metric spaces equipped with a group structure do not form group objects in the category of metric spaces. 
    Thus, it is not evident that restricting attention to stacks valued in groupoids is the most appropriate choice. 
    Rather, typical examples of metric spaces with a group structure (for instance \(SO(N)\) or the group of self-isometries of a homogeneous metric space) are understood as monoidal groups with respect to the \(\ell^1\)-product. 
    Hence, in order to treat actions of such objects on metric spaces within a stack-theoretic framework, it seems necessary to develop the theory of stacks as simplicial sheaves with effective descent condition, like \cite[{the paragraph below of Remark~5.14}]{Jardine-local-homotopy} (see also \cref{rmk-quot-stack-sset}).
    This is also why we did not require representability of diagonals.
  \end{enumerate}
  Although it is not entirely clear what kind of theory of stacks Gromov actually had in mind when he suggested such a possibility in \cite[{3.11 {\small\(\frac{3}{4}_+\)}}]{Gromov}, in view of the discussion in (2) above, it seems that any such theory can be realized only partially. 
  If Gromov intended to treat a geometric object that keeps track of isometries as a metric space---namely, if one were to impose representability of the diagonal morphism in the definition of a ``metric stack''---then such a theory would most likely fail.
  In the theory of algebraic stacks, the requirement that the diagonal be representable by an algebraic space is imposed for several reasons, one of which is that the categories of schemes or algebraic spaces do not behave well from a categorical viewpoint. 
  By contrast, the category of metric spaces enjoys substantially better formal properties than those categories. 
  For this reason, rather than imposing representability of diagonals, the author believes that one should instead pursue an approach in which the entire theory of metric spaces is extended to general sheaves on \(\tbMet_{\lhftop}\).
\end{remark}

Let us now study the fundamental properties of the notions defined above.

\begin{lemma}\label{prop-Hfib-for-stacks}
  The following assertions hold:
  \begin{assertion}
    \item \label{prop-Hfib-for-stacks-met}
    Let \(f:X\to Y\) be a morphism in \(\epMet\). 
    Then, \(f\) is a \Hfib as a morphism of extended pseudometric spaces if and only if \(f\) is a \fHfib as a morphism of stacks in groupoids over \(\tbMet_{\lhftop}\).
    \item \label{prop-Hfib-for-stacks-comp}
    Let \(f:\mcX\to \mcY\) and \(g:\mcY\to \mcZ\) be morphisms of stacks in groupoids over \(\tbMet_{\lhftop}\). 
    If \(f\) and \(g\) are \fHfibs, then \(g\circ f\) is a \fHfib.
    \item \label{prop-Hfib-for-stacks-comp-converse}
    Let \(f:\mcX\to \mcY\) and \(g:\mcY\to \mcZ\) be morphisms of stacks in groupoids over \(\tbMet_{\lhftop}\). 
    If \(f_*:\mcX(*)\to \mcY(*)\) is essentially surjective, and \(g\circ f\) is a \fHfib, then \(g\) is a \fHfib.
    \item \label{prop-Hfib-for-stacks-bc}
    Let \(f:\mcX\to \mcY\) and \(g:\mcY'\to \mcY\) be morphisms of stacks in groupoids over \(\tbMet_{\lhftop}\). 
    If \(f\) is a \fHfib, then the natural projection \(p:\mcX\times_{\mcY}\mcY'\to \mcY'\) \confer{\cref{notation-stacks} \ref{notation-stacks-2-prod}} is a \fHfib.
  \end{assertion}
\end{lemma}

\begin{proof}
  \Cref{prop-Hfib-for-stacks-met} follows immediately from \cref{prop-Hfib-char} ``\ref{prop-Hfib-char-Hfib}\(\Leftrightarrow\)\ref{prop-Hfib-char-lift}''.
  \Cref{prop-Hfib-for-stacks-comp,prop-Hfib-for-stacks-bc} follow by a straightforward verification once one unpacks the definitions.
  
  Next, we prove \cref{prop-Hfib-for-stacks-comp-converse}.
  Let \(r > s > 0\) be positive real numbers; \(y:*\to \mcY\) and \(c:2_s\to \mcZ\) morphisms of stacks in groupoids over \(\tbMet_{\lhftop}\); \(\alpha:g\circ y\tosim c\circ \iota_{0s}\) a 2-morphism.
  Since \(f\) is essentially surjective, there exists a morphism \(x:*\to \mcX\) and a 2-morphism \(\alpha':f\circ x\tosim y\).
  Write \(\alpha_0\dfn \alpha\circ (g\star\alpha'):(g\circ f)\circ x = g\circ (f\circ x) \tosim g\circ y\tosim c\circ \iota_{0s}\).
  Since \(g\circ f\) is a \fHfib, there exist a morphism \(a:2_r\to \mcX\), 2-morphism \(\beta:a\circ\iota_{0r}\tosim x\), and 2-morphism \(\gamma:(g\circ f)\circ a\tosim c\circ \iota_{rs}\) such that \(\alpha_0\circ ((g\circ f) \star \beta) = \gamma\star \iota_{0r}\). 
  Write \(b\dfn f\circ a:2_r\to \mcY\) and \(\beta'\dfn \alpha'\circ (f\star \beta): b\circ \iota_{0r}\tosim y\).
  Then, it holds that 
  \[
    \gamma\star\iota_{0r} = \alpha_0\circ ((g\circ f) \star \beta) 
    = \alpha\circ (g\star\alpha')\circ ((g\circ f)\star \beta)
    = \alpha\circ (g\star (\alpha'\circ (f\star \beta)))
    = \alpha\circ (g\star \beta').
  \]
  This implies that \(g\) is a \fHfib.
  This completes the proof of \cref{prop-Hfib-for-stacks} \ref{prop-Hfib-for-stacks-comp-converse}. 
\end{proof}

\begin{lemma}\label{prop-naive-metricity-Hfib-descending}
  Let \(f:\mcX\to \mcY\) be a morphism of stacks in groupoids over \(\tbMet_{\lhftop}\). 
  Assume that \(f\) is a \fHfib and that \(\mcX\) is a naive metric stack.
  Then, \(\mcY\) is a naive metric stack.
\end{lemma}

\begin{proof}
  \Cref{prop-naive-metricity-Hfib-descending} follows immediately from \cref{def-naive-metric-stack} and \cref{prop-Hfib-for-stacks} \ref{prop-Hfib-for-stacks-comp}.
\end{proof}

\begin{lemma}\label{prop-nm-stack-gd-moduli}
  Let \(\mcX\) be a naive metric stack.
  Then, the morphism of stacks \(\pi_{\mcX}:\mcX\to \mu(\mcX)\) is a \fHfib.
  In particular, \(\mu(\mcX)\) is a pseudometric space (i.e., the distance on \(\mu(\mcX)\) cannot have \(\infty\)).
\end{lemma}

\begin{proof}
  Since \(\mcX\) is a naive metric stack, there exist a metric space \(U\in\Met\) and a \fHfib \(p:U\to \mcX\). 
  Hence, by \cref{prop-Hfib-for-stacks} \ref{prop-Hfib-for-stacks-met} \ref{prop-Hfib-for-stacks-comp-converse}, to prove that \(\pi_{\mcX}:\mcX\to \mu(\mcX)\) is a \fHfib, it suffices to prove that \(\pi_{\mcX}\circ p:U\to \mu(\mcX)\) is a \Hfib.
  Note that, by \cref{prop-good-moduli-structure} \ref{prop-good-moduli-structure-uset} and \cref{def-Hfib-for-stacks} \ref{def-Hfib-for-stacks-ess-surj}, the map \(\pi_{\mcX}\circ p\) is surjective.
  Therefore, \(\mu(\mcX)\in \PMet\). 
  Write \(\rho:\mcX(*)/\!\simeq\,\tosim\uSet{\mu(\mcX)}\) for the natural bijection \confer{\cref{prop-good-moduli-structure}}.

  Let \(u\in U\) and \(x\in \mu(\mcX)\) be elements; \(\ep>0\) a positive real number.
  By \cref{prop-good-moduli-structure} \ref{prop-good-moduli-structure-dist}, there exists an integer \(N\in \N\), elements \(s_i\in \bRp\) (\(i\in\{0,\cdots,N\}\)), and morphisms \(f_i:2_{s_i}\to \mcX\) (\(i\in\{0,\cdots,N\}\)) such that the following conditions hold:
  \begin{condition}
    \item \label{prop-good-is-coarse-hfib-condi-1}
    \(\sum_{i=0}^N s_i \leq d_{\mu(\mcX)}((\pi_{\mcX}\circ p)(u), x) + \ep/2\). 
    \item \label{prop-good-is-coarse-hfib-condi-2}
    It holds that \(\rho(\iota_0^*f_0) = (\pi_{\mcX}\circ p)(u)\) and \(\rho(\iota_{s_N}^*f_N) = x\), where we regard \(f_0\) (resp. \(f_N\)) as an object of \(\mcX(2_{s_0})\) (resp. \(\mcX(2_{s_N})\)) via the 2-Yoneda lemma \confer{\cite[\href{https://stacks.math.columbia.edu/tag/0GWI}{Tag 0GWI}]{stacks-project}}.
    \item \label{prop-good-is-coarse-hfib-condi-3}
    For any \(i\in\{1,\cdots, N\}\), it holds that \(\iota_{s_{i-1}}^*f_{i-1} \cong \iota_0^*f_i\) in \(\mcX(*)\). 
  \end{condition}
  Since \(\mu(\mcX)\in\PMet\), it follows from \cref{prop-good-is-coarse-hfib-condi-1} that for each \(i\in\{0,\cdots,N\}\), \(s_i\neq \infty\).
  Moreover, by \cref{prop-good-is-coarse-hfib-condi-2} (and the 2-Yoneda lemma), there exists a 2-morphism \(\pi_{\mcX}\circ u\tosim f_0\circ\iota_{0s_0}\) (where we regard \(u\in U\) as a morphism \(*\to U\)).
  Write \(r_i\dfn s_i + \frac{\ep}{4(N+1)}\) (\(i\in\{0,\cdots,N\}\)). 
  Since \(p:U\to \mcX\) is a \fHfib, it follows from \cref{prop-good-is-coarse-hfib-condi-2} that there exists a morphism \(\tilde{f}_0:2_{r_0}\to U\) such that \(\tilde{f}(0) = u\) and \(p\circ \tilde{f}_0\cong f_0\circ \iota_{r_0s_0}\) in \(\mcX(2_{r_0})\) (via 2-Yoneda lemma). 
  Using \cref{prop-good-is-coarse-hfib-condi-3} together with the fact that \(p\) is a \fHfib, inductively we obtain, for each \(i\in\{1,\cdots, N\}\), a morphism \(\tilde{f}_i:2_{r_i}\to U\) such that \(\tilde{f}_{i-1}(r_{i-1}) = \tilde{f}_i(0)\) (\(\forall i\in\{1,\cdots, N\}\)), and \(p\circ \tilde{f}_i\cong f_i\circ \iota_{r_is_i}\) (\(\forall i\in\{1,\cdots, N\}\)).
  Write \(u'\dfn \tilde{f}_N(r_N)\in U\). 
  Then, by \cref{prop-good-is-coarse-hfib-condi-2}, it holds that \((\pi_{\mcX}\circ p)(u') = x\).
  Moreover, by \cref{prop-good-is-coarse-hfib-condi-3} and the definition of \(r_i\) (\(i\in\{0,\cdots,N\}\)), it holds that 
  \begin{align*}
    d_U(u,u') &\leq \sum_{i=0}^N r_i = \sum_{i=0}^N \left(s_i+\frac{\ep}{4(N+1)}\right) \\
    &\leq d_{\mu(\mcX)}((\pi_{\mcX}\circ p)(u), x) + \ep/2 + \ep/4 \\
    &< d_{\mu(\mcX)}((\pi_{\mcX}\circ p)(u), x) + \ep.
  \end{align*}
  By \cref{prop-Hfib-char} ``\ref{prop-Hfib-char-Hfib}\(\Leftrightarrow\)\ref{prop-Hfib-char-lift}'', this implies that \(\pi_{\mcX}\circ p\) is a \Hfib.
  This completes the proof of \cref{prop-nm-stack-gd-moduli}.
\end{proof}

To study the separatedness of good moduli spaces, we introduce the notion of separatedness for naive metric stacks.

\begin{definition}
  Let \(\mcX\) be a naive metric stack.
  We say that \(\mcX\) is \textbf{separated} if there exists a \fHfib \(p:X\to \mcX\) from a metric space \(X\in \Met\) such that for any object \(x\in \mcX(*)\), the subset \(p^{-1}(x)\dfn \left\{x'\in X\,\middle|\,p(x')\cong x\right\}\subset X\) is closed.
\end{definition}

\begin{lemma}\label{prop-sep-stack-char}
  Let \(\mcX\) be a naive metric stack.
  Then, the following assertions are equivalent:
  \begin{assertion}
    \item \label{prop-sep-stack-char-sep}
    \(\mcX\) is separated.
    \item \label{prop-sep-stack-char-gd-moduli-in-met}
    The good moduli space \(\mu(\mcX)\) is a metric space.
    \item \label{prop-sep-stack-char-fiber-cl}
    For any \fHfib \(p:X\to \mcX\) such that \(X\in \Met\) and any object \(x\in\mcX(*)\), the subset \(p^{-1}(x)\dfn \left\{x'\in X\,\middle|\,p(x')\cong x\right\}\subset X\) is closed.
  \end{assertion}
\end{lemma}

\begin{proof}
  The implication ``\ref{prop-sep-stack-char-fiber-cl}\(\Rightarrow\)\ref{prop-sep-stack-char-sep}'' follows immediately.
  We prove that ``\ref{prop-sep-stack-char-sep}\(\Rightarrow\)\ref{prop-sep-stack-char-gd-moduli-in-met}''.
  Assume that \(\mcX\) is a separated naive metric stack.
  Let \(p:X\to \mcX\) be a \fHfib from a metric space \(X\in\Met\) such that for any element \(x\in\mcX(*)\), the subset \(p^{-1}(\tilde{x})\dfn \left\{x'\in X\,\middle|\,p(x')\cong \tilde{x}\right\}\subset X\) is closed.
  By \cref{prop-nm-stack-gd-moduli}, \(\mu(\mcX)\) is a pseudometric space.
  Let \(x,x'\in \mu(\mcX)\) be distinct points.
  Then, by assumption, \((\pi_{\mcX}\circ p)^{-1}(x)\) and \((\pi_{\mcX}\circ p)^{-1}(x')\) are disjoint closed subsets of \(X\). 
  By \cref{prop-Hfib-for-stacks} \ref{prop-Hfib-for-stacks-comp} and \cref{prop-nm-stack-gd-moduli}, \(\pi_{\mcX}\circ p\) is a \fHfib.
  Hence, to prove that \(d_{\mu(\mcX)}(x,x')\neq 0\), it suffices to prove the following claim:
  \begin{claim}
    Let \(f:X\to Y\) be a \Hfib from \(X\in\Met\) to \(Y\in\PMet\) and \(y,y'\in Y\) distinct points.
    Assume that \(f^{-1}(y')\) is a closed subset of \(X\).
    Then, \(d_Y(y,y') > 0\). 
  \end{claim}
  \noindent
  Assume that \(d_Y(y,y')=0\).
  Let \(x\in f^{-1}(y)\) be a point and \(n\in\N\). 
  Since \(f\) is a \Hfib, there exists a point \(x'_n\in f^{-1}(y')\) such that \(d_X(x,x'_n) < 1/2^n\). 
  Then, \((x'_n)_{n\in \N}\) is a sequence in \(f^{-1}(y')\) and converges to \(x\in f^{-1}(y)\).
  Since \(f^{-1}(y')\) is closed, this implies that \(x\in f^{-1}(y')\). 
  In particular, \(y = f(x) = y'\). 
  This contradicts our assumption that \(y\) and \(y'\) are distinct. 
  This completes the proof of the claim and the implication ``\ref{prop-sep-stack-char-sep}\(\Rightarrow\)\ref{prop-sep-stack-char-gd-moduli-in-met}''.

  Next, we prove that ``\ref{prop-sep-stack-char-gd-moduli-in-met}\(\Rightarrow\)\ref{prop-sep-stack-char-fiber-cl}''.
  Assume that \(\mu(\mcX)\) is a metric space. 
  Let \(p:X\to \mcX\) be a \fHfib from a metric space \(X\in \Met\) and \(x\in\mcX(*)\) be an object. 
  Then, by \cref{prop-Hfib-for-stacks} \ref{prop-Hfib-for-stacks-comp} and \cref{prop-nm-stack-gd-moduli}, \(\pi_{\mcX}\circ p\) is a \Hfib.
  Since \(\mu(\mcX)\) is a metric space, \(\{\pi_{\mcX}(x)\}\subset \mu(\mcX)\) is closed. 
  Hence, \(p^{-1}(x) = (\pi_{\mcX}\circ p)^{-1}(\pi_{\mcX}(x))\subset X\) is closed.
  This completes the proof of \cref{prop-sep-stack-char}.
\end{proof}

\begin{corollary}\label{prop-good-is-coarse}
  Let \(\mcX\) be a separated naive metric stack.
  Then, the good moduli space \(\pi_{\mcX}:\mcX\to \mu(\mcX)\) of \(\mcX\) is a coarse moduli space. 
\end{corollary}

\begin{proof}
  \Cref{prop-good-is-coarse} follows immediately from \cref{prop-nm-stack-gd-moduli} and \cref{prop-sep-stack-char}.
\end{proof}

At the end of this section, for the convenience of the reader, we briefly review some general facts about quotient stacks.

\begin{definition}\label{def-quot-stack}
  Let \(\mcR\tto \mcU\) be a groupoid object in \(\Sh_{\lhftop}(\tbMet)\). 
  We write \([\mcU/\!_p\mcR]\) for the category fibered in groupoids corresponding, via the Grothendieck construction, to the presheaf of groupoids obtained by the groupoid object \(\mcR\tto \mcU\) in \(\Sh_{\lhftop}(\tbMet)\).
  We write \([\mcU/\mcR]\) for the stackification of \([\mcU/\!_p\mcR]\). 
  We refer to \([\mcU/\mcR]\) as the \textbf{quotient stack} associated with \(\mcR\tto \mcU\).
  Note that the stack \([\mcU/\mcR]\) is the 2-coequalizer of the diagram \(\mcR\tto \mcU\) in the 2-category of the stacks in groupoids over \(\tbMet_{\lhftop}\). 
\end{definition}

\begin{remark}\label{rmk-quot-stack-sset}
  As noted in \Cref{rmk-why-naive} \ref{rmk-why-naive-BG}, the author believes that a more appropriate definition of a metric stack should take values in simplicial sets. 
  The reason for this is its compatibility with the construction of classifying spaces via the bar construction.

  Suppose that a metric space \(G\) is equipped with a monoid structure \(m:G\otimes G\to G\) and \(e:*\to G\) with respect to the monoidal product (where we consider the monoidal structure on \(\epMet\) as the \(\ell^1\)-product) such that the underlying monoid \((G,m,e)\) is a group.
  Assume further that \(G\) acts on a metric space \(X\) through a morphism \(a: G\otimes X\to X\).
  From a homotopy-theoretic viewpoint, such an action would lead one to consider the stackification (injective fibrant replacement) of the simplicial object \(BG\dfn (G^{\otimes n}\otimes X)_{n\in \N}\) in \(\Sh_{\lhftop}(\tbMet)\) which arises from the bar construction applied to the action \(a:G\otimes X\to X\). 

  Moreover, \(BG\) is expected to serve as the classifying space of principal \(G\)-bundles. 
  To state this precisely, one needs an appropriate definition of a principal \(G\)-bundle. 
  Here, the author thinks it is effective, rather than starting from an abstract definition, to keep in mind the definition of \(BG\) via the bar construction and define principal \(G\)-bundles in a way that ensures \(BG\) indeed serves as their classifying space.
  For instance, consider a \(G\)-equivariant morphism of sheaves \(p:\mcP\to\mcF\) such that the \(G\)-action on \(\mcF\) is trivial, and moreover the induced left adjoint \(\sSh_{\lhftop}(\tbMet)^{BG}\to \sSh_{\lhftop}(\tbMet)_{/\mcF}\) (where \(\sSh_{\lhftop}(\tbMet)\) denotes the category of simplicial sheaves on \(\tbMet\) with respect to the \(\lhftop\)-topology and \(\sSh_{\lhftop}(\tbMet)^{BG}\) denotes the category of simplicial sheaves on \(\tbMet\) with a \(BG\)-action), perhaps this left adjoint should be regarded as a left adjoint in the \(\sSh_{\lhftop}(\tbMet)\)-enriched sense, is left exact. 
  If one calls such a morphism \(p\) \textbf{a principal \(G\)-bundle over \(\mcF\)}, then, by an analogue of Diaconescu’s theorem (cf.~e.g.~\cite[{Theorem~2.2}]{Moerdijk-classifying}) for classifying topoi, it is reasonable to expect that \(BG\) serves, in the topos-theoretic sense, as the classifying space of principal \(G\)-bundles.
  However, the primary aim of the present paper is the construction and investigation of fine moduli for compact metric spaces. 
  We therefore refrain from pursuing possible higher-categorical enhancement in this direction. 
\end{remark}

\begin{lemma}\label{prop-quot-stack-presentation}
  Let \(\mcX\) be a stack in groupoids over \(\tbMet_{\lhftop}\), \(\mcU\) a sheaf on \(\tbMet_{\lhftop}\), and \(f:\mcU\to \mcX\) a morphism of stacks in groupoids over \(\tbMet_{\lhftop}\). 
  Define \(\mcR\dfn \mcU\times_{f,\mcX,f}\mcU\), where we note that, since \(\mcX\) is a stack in groupoids over \(\tbMet_{\lhftop}\), the fiber product \(\mcR\) is a sheaf with respect to the \(\lhftop\)-topology.
  Write \(\pr_1,\pr_2:\mcR\tto \mcU\) for the natural projections and \(\eta:f\circ \pr_1\tosim f\circ \pr_2\) for the natural 2-morphism.
  Then, the following assertions hold:
  \begin{assertion}
    \item \label{prop-quot-stack-presentation-ff}
    The natural morphism \(\bar{f}: [\mcU/\mcR]\to \mcX\) of stacks in groupoids over \(\tbMet_{\lhftop}\) induced by \((\pr_1, \pr_2, \eta)\) is fully faithful. 
    \item \label{prop-quot-stack-presentation-eq}
    The natural morphism \(\bar{f}: [\mcU/\mcR]\to \mcX\) of stacks in groupoids over \(\tbMet_{\lhftop}\) induced by \((\pr_1, \pr_2, \eta)\) is an equivalence if and only if for any object \(T\in \tbMet\) and any morphism \(p:T\to \mcX\), there exist an \covlhf \(\{q_i:U_i\to T\}_{i\in I}\), a family of morphisms \(\{p_i:U_i\to \mcU\}_{i\in I}\), and a family of 2-morphisms \(\{\alpha_i:f\circ p_i\tosim p\circ q_i\}_{i\in I}\). 
  \end{assertion}
\end{lemma}

\begin{proof}
  \Cref{prop-quot-stack-presentation-ff} follows immediately from abstract nonsense in the theory of stacks \confer{\cite[\href{https://stacks.math.columbia.edu/tag/04WQ}{Tag 04WQ}]{stacks-project}}, once one recalls the definition of the sheaf \(\mcR\). 
  \Cref{prop-quot-stack-presentation-eq} follows immediately from \cref{prop-quot-stack-presentation-ff}, together with further abstract nonsense in the theory of stacks \confer{\cite[\href{https://stacks.math.columbia.edu/tag/046N}{Tag 046N}]{stacks-project}}. 
  This completes the proof of \cref{prop-quot-stack-presentation}.
\end{proof}

\begin{definition}\label{def-of-quot-type}
  Let \(\mcX\) be a naive metric stack.
  We say that \(\mcX\) is \textbf{of quotient type} if there exists a \fHfib \(p:U\to \mcX\) from a metric space \(U\in\Met\) such that for any object \(T\in\tbMet\) and any object \(t\in\mcX(T)\), there exists an \covlhf \(\{q_i:U_i\to T\}_{i\in I}\) such that for each index \(i\in I\), the object \(q_i^*t\in\mcX(U_i)\) lies in the essential image of \(U(U_i)\to \mcX(U_i)\). 
  Note that by \cref{prop-quot-stack-presentation} \ref{prop-quot-stack-presentation-ff} \ref{prop-quot-stack-presentation-eq}, \(\mcX\) is a naive metric stack of quotient type if and only if there exists a \fHfib \(p:U\to \mcX\) from a metric space \(U\in\Met\) such that the induced morphism \([U/\mcR]\to \mcX\), where \(\mcR\dfn U\times_{\mcX}U\in\Sh_{\lhftop}(\tbMet)\), is an equivalence of stacks in groupoids over \(\tbMet_{\lhftop}\). 
\end{definition}

\section{The Gromov--Hausdorff Stack}
\label{section: GH}

In this section, we define a moduli stack of \(n\)-pointed compact metric spaces \(\stGH_n\) and investigate its fundamental properties. 
In particular, we prove that \(\stGH\) is a naive metric stack \confer{\cref{thm: stGH is a naive metric stack}} whose coarse moduli space is isometric to the usual Gromov--Hausdorff space \(\topGH\) \confer{\cref{prop-metGH-is-usual-GH}}, and furthermore, that \(\stGH\) can be realized as a quotient stack of the space of metric spaces by a certain sheaf \confer{\cref{prop-stGH-quot}}.
Finally, we prove the statement asserting that \(\stGH\) is a ``fine moduli'' of compact metric spaces \confer{\cref{prop-fine-moduli-GH}}.

\begin{definition}\label{def-stGH}
  \
  \begin{enumerate}
    \item \label{def-stGH-stGH}
    We define a category \(\stGH\) as the full subcategory of \(\stAllMet\) \confer{\cref{def-stack-all-met}} consisting of those proper \Hfib \(p:P\to X\) whose target \(X\) lies in \(\tbMet\). 
    We write \(\pi:\stGH\to \tbMet\) for the composite of \(\stGH\subset \stAllMet\xto{\cod}\tbMet\). 
    Note that by \cref{prop-Hfib-basic} \ref{prop-Hfib-basic-bc}, \(\pi:\stGH\to\tbMet\) is a category fibered in groupoids.
    \item
    Let \(n\in \N\).
    We write \(\stGH_n\) for the category defined as follows:
    \begin{itemize}
      \item
      An object of \(\stGH_n\) is the pair \((p:P\to X, (s_i:X\to P)_{i=1}^n)\) of an object \(p:P\to X\) of \(\stGH_n\) and sections \(s_i:X\to P\) (\(i\in \{1,\cdots,n\}\)) of \(p\).
      \item
      A morphism
      \[(p':P'\to X', (s_i':X'\to P')_{i=1}^n)\to (p:P\to X, (s_i:X\to P)_{i=1}^n)\]
      in \(\stGH_n\) is a morphism \((g:P'\to P, f:X'\to X)\) in \(\stGH\) such that for each \(i\in\{1,\cdots,n\}\), it holds that \(g\circ s_i' = s_i\circ f\).
    \end{itemize}
    We regard \(\stGH_n\) as a category fibered in groupoids over \(\Met\) by the functor \((p:P\to X, (s_i:X\to P)_{i=1}^n)\mapsto X\).
    Note that \(\stGH_0 = \stGH\).
    \item
    Let \(n\in \N\) be an integer.
    We write
    \begin{align*}
      p^{\stGH}_n:\stGH_{n+1}&\to \stGH_n, \\
      (p, (s_1,\cdots,s_{n+1})) &\mapsto (p, (s_1,\cdots,s_n)).
    \end{align*}
    For each \(i\in \{1,\cdots,n\}\), we write
    \begin{align*}
      s_{i}^{\stGH}:\stGH_n &\to \stGH_{n+1}, \\
      (p, (s_1,\cdots,s_n)) &\mapsto (p, (s_1,\cdots,s_n,s_i)).
    \end{align*}
    Then, it holds that \(p^{\stGH}_n\circ s^{\stGH}_i = \id_{\stGH_n}\).
  \end{enumerate}
\end{definition}

\begin{lemma}\label{prop-proper-Hfib-is-lhf-local}
  Let \(B\in\tbMet\) be an object, \(\left\{f_i:U_i\to B\right\}_{i\in I}\) an \covlhf, and \(p:X\to B\) a morphism in \(\Met\). 
  Then, \(p\) is a proper \Hfib if and only if for each \(i\in I\), the natural projection \(p_i:X_i\dfn X\times_BU_i\to U_i\) is a proper \Hfib.
\end{lemma}

\begin{proof}
  Necessity follows immediately from \cref{prop-Hfib-basic} \ref{prop-Hfib-basic-bc}.
  Next, we prove sufficiency.
  Assume that for each \(i\in I\), the natural projection \(p_i:X_i\dfn X\times_BU_i\to U_i\) is a proper \Hfib.

  To prove that \(p\) is a \Hfib, let \(x\in X\) be a point, \(b\in B\) a point, and \(\ep > 0\) a positive real number.
  Since \(\left\{f_i:U_i\to B\right\}_{i\in I}\) is an \covlhf, it follows from \cref{defi: cov in Met} \ref{defi: cov in Met enumi: lhf condi: 3pt} that there exist an index \(i\in I\) and points \(u_0, u_1\in U_i\) such that \(f_i(u_0) = p(x)\), \(f_i(u_1) = b\), and \(d_{U_i}(u_0, u_1) < d_B(p(x), b) + \ep/2\). 
  Write \(\tilde{x}\dfn (x, u_0)\in X_i = X\times_BU_i\). 
  Since \(p_i:X_i\to U_i\) is a proper \Hfib, it follows from \cref{prop-Hfib-char} \ref{prop-Hfib-char-ineq} that there exists a point \(\tilde{x}_1\in X_i\) such that \(d_{X_i}(\tilde{x}, \tilde{x}_1) < d_{U_i}(u_0, u_1) + \ep/2\). 
  Write \(\tilde{x}'\in X\) for the image of \(\tilde{x}_1\in X_i\) under the natural morphism \(X_i = X\times_BU_i\to X\). 
  Then, it holds that 
  \[d_X(x, \tilde{x}') \leq d_{X_i}(\tilde{x}, \tilde{x}_1) < d_{U_i}(u_0, u_1) + \ep/2 < d_B(p(x), b) + \ep.\]
  This implies that \(p\) is a \Hfib.
  By applying \cref{prop: Hfib fiberwise cpt prop}, we conclude that \(p\) is proper.
  This completes the proof of \cref{prop-proper-Hfib-is-lhf-local}.
\end{proof}

\begin{corollary}\label{prop-stGHn-is-a-stack}
  For each \(n\in \N\), \(\stGH_n\) is a stack in groupoids over \(\tbMet_{\lhftop}\). 
\end{corollary}

\begin{proof}
  Since \(\stGH_n\) is a category fibered in groupoids over \(\tbMet\), \cref{prop-stGHn-is-a-stack} follows immediately from \cref{lem: lhf subcan} \ref{lem: lhf subcan ass: isom sh}, \cref{lem: eff descent}, and \cref{prop-proper-Hfib-is-lhf-local}.
\end{proof}

Next, we verify that \(\stGH\) is a naive metric stack.
To this end, we use a universal property of the Urysohn universal space.

\begin{definition}
  We write \(\Urys\) for the Urysohn universal space.
  The \textbf{Urysohn universal space} \(\Urys\) is a separable and complete metric space satisfying the following condition:
  \begin{condition*}
    \item
    For any isometric embedding \(i:A\inj X\) of metric spaces with \(\#(A),\#(X) < \infty\), and any isometric embedding \(f:A\inj \Urys\), there exists an isometric embedding \(\tilde{f}: X\inj \Urys\) such that \(f = \tilde{f}\circ i\).
  \end{condition*}
  Note that by \cite[{3.11.{\small \(\frac{2}{3}_+\)}, Corollary~(5)}]{Gromov}, the Urysohn universal space exists and is unique up to isometry in \(\Met\).
  We write \(\Cpt(\Urys)\dfn \Cpt(\Urys/*)\) \confer{\cref{def-inCpt}}.
\end{definition}

\begin{remark}\label{rmk-prop-Ury-univ}
  Let \(U\) be the Urysohn universal space, \(K\) a compact metric space, \(i:A\inj K\) a closed isometric embedding, and \(f:A\inj U\) an isometric embedding.
  Then, by \cite[{3.11.{\small \(\frac{2}{3}_+\)}, p.~89 Exercises~(b)}]{Gromov}, there exists an isometric embedding \(g:K\inj U\) such that \(g\circ i = f\). 
\end{remark}

\begin{theorem}\label{thm: stGH is a naive metric stack}
  For any object \(T\in\tbMet\), define 
  \begin{align*}
    p_T:\inCpt(\Urys/*)(T) &\to \stGH(T), \\
    [A\subset \Urys\times T] &\mapsto [A\to T].
  \end{align*}
  Then, the family of maps \((p_T)_{T\in\tbMet}\) defines a morphism \(p:\inCpt(\Urys/*)\to\stGH\) of stacks in groupoids over \(\tbMet_{\lhftop}\). 
  Moreover, the corresponding morphism \(\pi:\Cpt(\Urys)\to \stGH\), via the Yoneda lemma \confer{\cref{prop: cpt repble}}, is a \fHfib. 
  In particular, \(\stGH\) is a naive metric stack.
\end{theorem}

\begin{proof}
  The first assertion follows immediately from the definition of the family of maps \((p_T)_{T\in\tbMet}\).
  Let
  \[\begin{tikzcd}
    {*} \ar[r, "{[i]}"] \ar[d, hook, "{\iota_0}"'] & 
    {\Cpt(\Urys)} \ar[d, "{\pi}"] \ar[ld, Rightarrow, "{[\alpha]}"', shorten=4mm]  \\
    {2_r} \ar[r, "{[p]}"'] & {\stGH}
  \end{tikzcd}\]
  be a 2-commutative diagram.
  Write \(p:K\to 2_r\) for the proper \Hfib corresponding to \([p]\) via the 2-Yoneda lemma; 
  \(i:K_0\inj \Urys\) for an isometric embedding corresponding to \([i]\) via the Yoneda lemma \confer{\cref{prop: cpt repble}}; 
  \(\alpha:K_0\tosim p^{-1}(0)\) for the isometry corresponding to \([\alpha]\) via the 2-Yoneda lemma; 
  \(\tilde{\alpha}:K_0\inj K\) for the composite of \(\alpha\) and the inclusion \(p^{-1}(0)\subset K\). 
  By \cref{rmk-prop-Ury-univ}, there exists an isometric embedding \(j:K\to \Urys\) such that \(i = j\circ \tilde{\alpha}\).
  Since \(p\) is a \Hfib, the isometric embedding \(j\) corresponds, via the Yoneda lemma \confer{\cref{prop: cpt repble}}, to a morphism \([j]:2_r\to \Cpt(\Urys)\).
  Since \(i = j\circ \tilde{\alpha}\), it holds that \([i] = [j]\circ \iota_0\). 
  Write \(\gamma:K\tosim \im((j,p):K\to \Urys\times 2_r)\) for the induced isomorphism.
  Then, by the definitions of \(\pi\) and \([j]\), the isomorphism \(\gamma\) corresponds, via the 2-Yoneda lemma, to a 2-morphism \([\gamma]:\pi\circ [j]\tosim [p]\) such that \([\gamma]\star \iota_0 = [\alpha]\). 
  In particular, \(\pi\) is a \fHfib.
  This completes the proof of \cref{thm: stGH is a naive metric stack}.
\end{proof}

\begin{definition}\label{def-Urys-chart}
  We call the morphism \(\pi_{\Urys}:\Cpt(\Urys)\to \stGH\) of stacks in groupoids over \(\tbMet_{\lhftop}\), constructed in \cref{thm: stGH is a naive metric stack}, the \textbf{Urysohn chart}. 
\end{definition}

Next, we verify the good moduli space of \(\stGH\) is isometric to the usual Gromov--Hausdorff space.

\begin{definition}
  We write \(\pi_{\stGH}:\stGH\to \metGH\) for the good moduli space of \(\stGH\).
\end{definition}

\begin{theorem}\label{prop-metGH-is-usual-GH}
  The following assertions hold:
  \begin{assertion}
    \item \label{prop-metGH-is-usual-GH-cpt-isomet}
    A point of \(\metGH\) corresponds to an isometry class of compact metric spaces.
    For any compact metric space \(C\), we write \([C]\in \metGH\) for the point corresponding to the isometry class of \(C\). 
    \item \label{prop-metGH-is-usual-GH-dist-GH}
    For any compact metric spaces \(C\) and \(D\), the distance \(d_{\metGH}([C], [D])\) is equal to the Gromov--Hausdorff distance \confer{\cite[{Definition~7.3.10}]{BBI}} between \(C\) and \(D\). 
  \end{assertion}
  In particular, \(\metGH\) is isometric to the Gromov--Hausdorff space \confer{\cite[{after the proof of Theorem~7.3.30}]{BBI}}.
\end{theorem}

\begin{proof}
  \Cref{prop-metGH-is-usual-GH-cpt-isomet} follows immediately from \cref{prop-good-moduli-structure} \ref{prop-good-moduli-structure-uset} and \cref{def-stGH} \ref{def-stGH-stGH}.
  
  Next, we prove \cref{prop-metGH-is-usual-GH-dist-GH}. 
  Write \(d_0\) for the Gromov--Hausdorff distance (in the sense of \cite[{Definition~7.3.10}]{BBI}).
  Let \(X, Y\) be compact metric spaces and \(\ep > 0\) a positive real number.
  By \cite[Theorem~7.3.25]{BBI}, there exists a correspondence \(R\subset X\times Y\) \confer{\cite[Definition~7.3.17]{BBI}} such that \(\mathrm{dis}(R) \leq d_0(X, Y) + \ep\), where \(\mathrm{dis}(R)\) is the distorsion of \(R\) \confer{\cite[Definition 7.3.21]{BBI}}.
  Write \(r_0\dfn \frac{1}{2}\cdot \mathrm{dis}(R)\); 
  \(d_Z\) for the distance on the set \(Z\dfn X\sqcup Y\) determined by the condition that \(d_Z|_X = d_X\), \(d_Z|_Y = d_Y\), and for any \(x\in X\), any \(y\in Y\), 
  \[d_Z(x,y)=d_Z(y,x) \dfn r_0 + \inf_{(x',y')\in R}(d_X(x,x') + d_Y(y,y')).\]
  Note that by the proof of \cite[Theorem~7.3.25]{BBI}, \(d_Z\) is a distance on \(Z\).
  For any \(x\in X\) and any \(y\in Y\), it holds that \(d_Z(x,Y) = d_Z(X, y) = r_0\).
  Hence, the natural map \(p:Z\to 2_{r_0}\) determined by \(p(X) = \{0\}\) and \(p(Y) = \{r_0\}\) is a proper \Hfib.
  In particular, there exists a morphism \(2_{r_0}\to \stGH\) such that \([X] = [\pi_{\stGH}(p^{-1}(0))]\), and \([Y] = [\pi_{\stGH}^{-1}(r_0)]\).
  This implies that \(d_{\metGH}([X], [Y]) \leq r_0 \leq d_0(X, Y) + \ep\).
  By letting \(\ep \to 0\), we conclude that \(d_{\metGH}([X], [Y]) \leq d_0(X, Y)\). 

  Next, we prove that \(d_{\metGH}([X], [Y]) \geq d_0(X, Y)\).
  Write \(r_1\dfn d_{\metGH}([X], [Y])\) and \(f:2_{r_1}\to \metGH\) for the morphism determined by \(f(0) \dfn [X]\) and \(f(r_1) \dfn [Y]\). 
  By \cref{prop-nm-stack-gd-moduli}, there exists a morphism of stacks in groupoids \(\tilde{f}:2_{r_1 + \ep}\to \stGH\) over \(\tbMet_{\lhftop}\) such that \(\pi_{\stGH}\circ \tilde{f} = f\circ \iota_{r_1+\ep, r_1}\).
  By the 2-Yoneda lemma, there exists a proper \Hfib \(q:Q\to 2_{r_1+\ep}\) such that \([q^{-1}(0)] = [X]\), and \([q^{-1}(r_1+\ep)] = [Y]\). 
  This implies that \(d_0(X, Y)\leq \dH_Q(q^{-1}(0), q^{-1}(r_1 + \ep)) = r_1 + \ep\). 
  By letting \(\ep \to 0\), we conclude that \(d_0(X, Y)\leq r_1 = d_{\metGH}([X], [Y])\). 
  This completes the proof of \cref{prop-metGH-is-usual-GH}.
\end{proof}

\begin{corollary}\label{prop-stGH-sep}
  \(\stGH\) is separated. 
  In particular, \(\pi_{\stGH}:\stGH\to \metGH\) is the coarse moduli space.
\end{corollary}

\begin{proof}
  \Cref{prop-stGH-sep} follows immediately from the fact that the Gromov--Hausdorff distance is actually a distance (cf. \cite[Theorem 7.3.30]{BBI}), together with \cref{prop-sep-stack-char} and \cref{prop-metGH-is-usual-GH}. 
\end{proof}

\begin{corollary}\label{prop-metGH-is-cplt}
  \(\metGH\) is complete.
\end{corollary}

\begin{proof}
  Since \(\stGH\) is a naive metric stack, it follows from \cref{prop-nm-stack-gd-moduli} that the natural morphism \(\pi_{\stGH}:\stGH\to \metGH\) to the good moduli space is a \fHfib.
  Hence, by \cref{prop-Hfib-for-stacks} \ref{prop-Hfib-for-stacks-comp}, the composite \(\pi_{\stGH}\circ\pi_{\Urys}:\Cpt(\Urys)\to \metGH\) is a \Hfib.
  Since the Urysohn universal space \(\Urys\) is complete, it follows from \cref{prop-hausdorff-dist-cplt} that \(\Cpt(\Urys)\) is complete. 
  Hence, by \cref{prop-Hfib-cplt-descending}, \(\metGH\) is complete. 
  This completes the proof of \cref{prop-metGH-is-cplt}. 
\end{proof}

Next, we verify that \(\stGH\) is a naive metric stack of quotient type, thereby giving a partial answer to what Gromov suggested in \cite[{3.11 {\small\(\frac{3}{4}_+\)}}]{Gromov}.

\begin{lemma}\label{prop-CptUry-lift}
  Let \(T\) be a totally bounded metric space and \(p:T\to \stGH\) a morphism of stacks in groupoids over \(\tbMet_{\lhftop}\).
  Then, there exists a morphism \(q:T\to \Cpt(\Urys)\) and a 2-morphism \(\alpha:\pi_{\Urys}\circ q\tosim p\) \confer{\cref{def-Urys-chart}}. 
\end{lemma}

\begin{proof}
  Write \(f:P\to T\) for the proper \Hfib corresponding, via the 2-Yoneda lemma, to the morphism \(p:T\to \stGH\).
  Then, by \cref{prop-Hfib-tot-bd}, \(P\) is totally bounded.
  Hence, by \cref{rmk-prop-Ury-univ}, there exists an isometric embedding \(i: P\subset \hat{P}\inj \Urys\).
  Write \(\tilde{f}\dfn (p,i): P\to T\times \Urys\).
  Then, the subset \(\im(\tilde{f})\subset T\times \Urys\) is an element of \(\inCpt(\Urys/*)(T)\).
  Write \(q:T\to \Cpt(\Urys)\) for the morphism corresponding, via the Yoneda lemma, to the subset \(\im(\tilde{f})\subset T\times \Urys\).
  Then, the isometry \(P\tosim \im(\tilde{f})\) induces a 2-morphism \(\pi\circ q\tosim p\).
  This completes the proof of \cref{prop-CptUry-lift}.
\end{proof}

\begin{corollary}\label{prop-stGH-quot}
  Write \(\mcR\dfn \Cpt(\Urys)\times_{\pi,\stGH,\pi}\Cpt(\Urys)\in \Sh_{\lhftop}(\tbMet)\).
  Then, the natural morphism \([\Cpt(\Urys)/\mcR]\to \stGH\) \confer{\cref{def-quot-stack}} is an equivalence.
  In particular, \(\stGH\) is a naive metric stack of quotient type.
\end{corollary}

\begin{proof}
  \Cref{prop-stGH-quot} follows immediately from \cref{prop-quot-stack-presentation} \ref{prop-quot-stack-presentation-eq} and \cref{prop-CptUry-lift}.
\end{proof}

Next, we verify that for each \(n\in\N\), \(\stGH_n\) is a naive metric stack of quotient type.
To this end, we introduce a class of morphisms of stacks that are represented by proper submetries.

\begin{definition}\label{def-rep-prop-Hfib}
  Let \(f:\mcX\to \mcB\) be a morphism of stacks in groupoids over \(\tbMet_{\lhftop}\). 
  We say that \(f\) is a \textbf{representable proper \Hfib} if for any object \(T\in\tbMet_{\lhftop}\) and any morphism \(b:T\to\mcB\), there exist a proper \Hfib \(f_T:X_T\to T\) in \(\Met\) and an equivalence \(X_T\cong \mcX\times_{\mcB}T\) in \(\Stack_{\lhftop}(\tbMet)_{/T}\). 
\end{definition}

\begin{lemma}\label{prop-rep-proper-Hfib}
  The following assertions hold:
  \begin{assertion}
    \item \label{prop-rep-proper-Hfib-met}
    For any proper \Hfib \(f:X\to Y\) in \(\Met\), \(f\) is a representable proper \Hfib as a morphism of stacks in groupoids over \(\tbMet_{\lhftop}\). 
    \item \label{prop-rep-proper-Hfib-comp}
    Let \(f:\mcX\to \mcY\) and \(g:\mcY\to \mcZ\) be representable proper \Hfib of stacks in groupoids over \(\tbMet_{\lhftop}\). 
    Then, \(g\circ f\) is a representable proper \Hfib. 
    \item \label{prop-rep-proper-Hfib-bc}
    Let \(f:\mcX\to \mcY\) and \(g:\mcY'\to\mcY\) be morphisms of stacks in groupoids over \(\tbMet_{\lhftop}\). 
    Assume that \(f\) is a representable proper \Hfib.
    Then, the natural morphism \(f':\mcX'\dfn \mcX\times_{\mcY}\mcY'\to \mcY'\) is a representable proper \Hfib.
    \item \label{prop-rep-proper-Hfib-fhfib}
    A representable proper \Hfib \(f:\mcX\to \mcB\) is a \fHfib.
  \end{assertion}
\end{lemma}

\begin{proof}
  \Cref{prop-rep-proper-Hfib-met,prop-rep-proper-Hfib-comp,prop-rep-proper-Hfib-bc} follow by a straightforward verification once one unpacks the definitions.
  \Cref{prop-rep-proper-Hfib-fhfib} follows immediately from \cref{prop-Hfib-char} ``\ref{prop-Hfib-char-Hfib}\(\Leftrightarrow\)\ref{prop-Hfib-char-lift}'' and \cref{def-Hfib-for-stacks}.
\end{proof}

\begin{lemma}\label{prop-rep-proper-Hfib-in-met}
  Let \(p:\mcX\to B\) be a representable proper \Hfib of stacks in groupoids over \(\tbMet_{\lhftop}\) such that \(B\in\Met\). 
  Then, the natural morphism to the good moduli space \(\pi_{\mcX}:\mcX\to \mu(\mcX)\) is an equivalence.
  In other words, \(\mcX\) is representable by a metric space.
\end{lemma}

\begin{proof}
  To prove \cref{prop-rep-proper-Hfib-in-met}, it suffices to prove that for any object \(T\in\tbMet\) and any morphism \(b:T\to B\), the natural morphism \(\mcX\times_BT\to \mu(\mcX)\times_BT\) of stacks in groupoids over \(\tbMet_{\lhftop}\) is an equivalence.
  Since \(p:\mcX\to B\) is a representable proper \Hfib, it follows from \cref{prop-rep-proper-Hfib} \ref{prop-rep-proper-Hfib-bc} that we may assume without loss of generality that \(\mcX\times_BT\in\tbMet\). 
  Moreover, by \cref{prop-good-moduli-structure} \ref{prop-good-moduli-structure-uset}, the natural morphism \(\mcX\times_BT\to \mu(\mcX)\times_BT\) is bijective.
  Let \(x, x'\in \mu(\mcX)\times_BT\) be points and \(\ep > 0\) a positive real number.
  Write \(q:\mu(\mcX)\times_BT\to \mu(\mcX)\) for the natural projection.
  By \cref{prop-good-moduli-structure} \ref{prop-good-moduli-structure-dist}, there exist an integer \(N\in \N\), elements \(r_i\in \bRp\) (\(i\in\{0,\cdots,N\}\)), and morphisms \(f_i:2_{r_i}\to \mcX\) (\(i\in\{0,\cdots,N\}\)) such that the following conditions hold:
  \begin{condition}
    \item \label{prop-rep-proper-Hfib-in-met-1}
    \(\sum_{i=0}^N r_i \leq d_{\mu(\mcX)}(q(x),q(x')) + \ep\). 
    \item \label{prop-rep-proper-Hfib-in-met-2}
    It holds that \(\pi_{\mcX}(\iota_0^*f_0) = q(x)\), and \(\pi_{\mcX}(\iota_{r_N}^*f_N) = q(x')\), where we regard \(f_0\) (resp. \(f_N\)) as an object of \(\mcX(2_{r_0})\) (resp. \(\mcX(2_{r_N})\)) via the 2-Yoneda lemma \confer{\cite[\href{https://stacks.math.columbia.edu/tag/0GWI}{Tag 0GWI}]{stacks-project}}.
    \item \label{prop-rep-proper-Hfib-in-met-3}
    For any \(i\in\{1,\cdots, N\}\), it holds that \(\iota_{r_{i-1}}^*f_{i-1} \cong \iota_0^*f_i\) in \(\mcX(*)\). 
  \end{condition}
  Write \(T_0\dfn \{0\}\cup \left\{\sum_{i=0}^n r_i\,\middle|\,0\leq n\leq N\right\}\subset\R\). 
  Then, \(T_0\in\tbMet\).
  Moreover, by \cref{prop-rep-proper-Hfib-in-met-2,prop-rep-proper-Hfib-in-met-3}, the family of morphisms \(\{\pi_{\mcX}\circ f_i:2_{r_i}\to \mu(\mcX)\}_{i=0}^N\) defines a morphism \(T_0\to \mu(\mcX)\). 
  Write \(s:T_0\to \mu(\mcX)\times_BT_0\) for the section of the natural projection \(p_0:\mu(\mcX)\times_BT_0\to T_0\) determined by this morphism \(T_0\to \mu(\mcX)\). 
  For each \(k\in \{0,\cdots, N\}\), the morphism \(f_k:2_{r_k}\to \mcX\) defines a morphism \(\tilde{f}_k:2_{r_k}\to \mcX\times_BT_0\) such that \((p_0\circ \tilde{f}_k)(0) = \sum_{i=0}^{k-1}r_i\) and \((p_0\circ \tilde{f}_k)(r_k) = \sum_{i=0}^kr_i\) (where we regard \(\sum_{i=0}^{-1}r_i = 0\)).
  Since \(T_0\in\tbMet\), it follows from \cref{prop-rep-proper-Hfib} \ref{prop-rep-proper-Hfib-bc} that \(\mcX\times_BT_0\) is equivalent to an object of \(\tbMet\). 
  Thus, if we write \(R\dfn \sum_{i=0}^N r_i\), then the family of morphisms \(\left\{\tilde{f}_k:2_{r_k}\to \mcX\times_BT_0\right\}_{k=0}^N\) defines a morphism \(\tilde{f}:2_R\to \mcX\times_BT_0\) such that \((p_0\circ\tilde{f})(0) = 0\) and \((p_0\circ\tilde{f})(R)=R\).
  Write \(\tilde{x}, \tilde{x}'\in \mcX\times_BT\) for the points such that \((\pi_{\mcX}\times_B\id_{T_0})(\tilde{x}) = s(0)\) and \((\pi_{\mcX}\times_B\id_{T_0})(\tilde{x}') = s(R)\). 
  Then, it holds that \(\tilde{f}(0)=\tilde{x}\) and \(\tilde{f}(R) = \tilde{x}'\).
  In particular, by \cref{prop-rep-proper-Hfib-in-met-1}, \(d_{\mcX\times_BT}(\tilde{x}, \tilde{x}') \leq d_{\mu(X)\times_BT}(x, x') + \ep\). 
  By letting \(\ep\to 0\), we conclude that \(\pi_{\mcX}\times_B\id_T:\mcX\times_BT\to \mu(\mcX)\times_BT\) is an isomorphism.
  This completes the proof of \cref{prop-rep-proper-Hfib-in-met}.
\end{proof}

\begin{lemma}\label{prop-naive-metric-of-quot-and-rep-proper-Hfib}
  Let \(f:\mcX\to \mcB\) be a representable proper \Hfib.
  Assume that \(\mcB\) is a naive metric stack of quotient type. 
  Then, \(\mcX\) is also a naive metric stack of quotient type.
\end{lemma}

\begin{proof}
  Since \(\mcB\) is a naive metric stack of quotient type, there exists a \fHfib \(p:U\to \mcB\) from a metric space \(U\in\Met\) such that the morphism \([U/\mcR]\to \mcB\), where \(\mcR\dfn U\times_{\mcB}U\), induced by \(p\) is an equivalence. 
  Since \(f\) is a representable proper \Hfib, it follows from \cref{prop-rep-proper-Hfib-in-met} that there exist a proper \Hfib \(g:P\to U\), a morphism \(q:P\to \mcX\), and a 2-morphism \(\alpha:p\circ g\tosim f\circ q\) such that the morphism \(P\to\mcX\times_{\mcB}U\) induced by the triple \((g,q,\alpha)\) is an equivalence.
  Then, by \cref{prop-Hfib-for-stacks} \ref{prop-Hfib-for-stacks-bc}, \(q\) is a \fHfib. 
  In particular, \(\mcX\) is a naive metric stack.

  To prove that \(\mcX\) is of quotient type, let \(T\in\tbMet\) be an object and \(t:T\to \mcX\) a morphism of stacks in groupoids over \(\tbMet_{\lhftop}\). 
  Write \([t]\in\mcX(T)\) (resp. \([f\circ t]\in \mcB(T)\)) for the object corresponding, via the 2-Yoneda lemma, to the morphism \(t:T\to\mcX\) (resp. \(f\circ t:T\to \mcB\)).
  Since \(\mcB\) is of quotient type, it follows from \cref{prop-quot-stack-presentation} \ref{prop-quot-stack-presentation-eq} that there exist an \covlhf \(\left\{r_i:U_i\to T\right\}_{i\in I}\), a family of morphisms \(\{u_i:U_i\to U\}_{i\in I}\), and a family of 2-morphisms \(\{\alpha_i:f\circ t\circ r_i\tosim p\circ u_i\}_{i\in I}\). 
  Then, since the morphism \(P\to\mcX\times_{\mcB}U\) is an equivalence, for each index \(i\in I\), the triple \((t\circ r_i, u_i, \alpha_i)\) induces a morphism \(v_i:U_i\to P\) and a 2-morphism \(q\circ v_i \tosim t\circ r_i\) such that \(g\circ v_i = u_i\). 
  Thus, by \cref{prop-quot-stack-presentation} \ref{prop-quot-stack-presentation-eq}, \(\mcX\) is of quotient type.
  This completes the proof of \cref{prop-naive-metric-of-quot-and-rep-proper-Hfib}.
\end{proof}

\begin{lemma}\label{lem: univ family}
  Let \(n\in \N\) be an integer, 
  \(X\) a totally bounded metric space,
  \(p:P\to X\) a representable proper \Hfib, and 
  \((s_i:X\to P)_{i=1}^n\) a family of sections of \(p\).
  Then, there exist morphisms \([P]:X\to \stGH_n\) and \([\Delta]:P\to \stGH_{n+1}\) of stacks in groupoids over \(\tbMet_{\lhftop}\) such that the 
  following diagram is 2-Cartesian:
  \begin{equation}
    \label{lem: univ family eq: diag}
    \begin{tikzcd}
      {X} \ar[r, "{s_1,\cdots, s_n}"] \ar[d, "{[P]}"'] &[1.5cm]
      {P} \ar[r, "{p}"] \ar[d, "{[\Delta]}"]
      & {X} \ar[d, "{[P]}"] \\
      {\stGH_{n}} \ar[r, "{s_1^{\stGH},\cdots, s_n^{\stGH}}"] &
      {\stGH_{n+1}} \ar[r, "{p^{\stGH}_{n}}"] &
      {\stGH_{n}}.
    \end{tikzcd}
    \tag{\(\dagger\)}
  \end{equation}
\end{lemma}

\begin{proof}
  Write \(p':P\times_XP\to P\) for the first projection and \(s_i'\dfn s_i\times_X\id_P:P\to P\times_XP\) (\(i\in \{1,\cdots,n\}\)).
  Since the family \((p',(s_1',\cdots,s_n',\Delta_{P/X}))\) determines an object of \(\stGH_{n+1}\) over \(P\), it follows from the 2-Yoneda lemma that there exists a morphism \([\Delta]:P\to\stGH_{n+1}\) classifying the object \[(p\times_X\id_P, (s_1\times_X\id_P,\cdots,s_n\times_X\id_P,\Delta_{P/X}))\in\stGH_{n+1}(P).\]
  By the definition of the morphisms \(p^{\stGH}_{n}, s^{\stGH}_{i}\) (\(i\in \{1,\cdots,n\}\)), the diagram \cref{lem: univ family eq: diag} is (strictly) commutative.
  Thus, to complete the proof of \cref{lem: univ family}, it suffices to prove that the right small square in the diagram \cref{lem: univ family eq: diag} is 2-Cartesian.

  Write \(\mcP\dfn \stGH_{n+1}\times_{\stGH_n,[P]}X\) for the 2-fiber product and \(\tau:P\to \mcP\) for the natural morphism induced from \((p, [\Delta])\) (and the fact that \([P]\circ p = p_n^{\stGH}\circ [\Delta])\). 
  By the universality of the 2-fiber product, for each object \(T\in\tbMet\), the functor of groupoids \(\tau_T:P(T)\to \mcP(T)\) is fully faithful.
  Hence, to complete the proof of \cref{lem: univ family}, it suffices to prove that \(\tau\) is essentially surjective. 
  Let \(T\in\tbMet\) be an object and \(\bullet\in\mcP(T)\) an object. 
  Then, \(\bullet\) consists of the triple \((f, [q], \alpha)\) of a morphism \(f:T\to X\), an object \([q]=(q:Q\to T; t_1:T\to Q,\cdots, t_{n+1}:T\to Q)\), and an isomorphism \(\alpha:Q\tosim P\times_XT\) in \(\Met\) such that for each index \(i\in\{1,\cdots,n\}\), \(\alpha\circ t_i = s_i\times_X \id_T\). 
  Write \(\pr_P:P\times_XT\to P\) for the natural projection.
  Then, the isomorphism \(\alpha\) induces an isomorphism \(\bullet=(f,[q],\alpha)\tosim \tau_T(\pr_P\circ (s_{n+1}\times_X\id_T))\) in \(\mcP(T)\). 
  Thus, \(\tau\) is an equivalence. 
  This completes the proof of \cref{lem: univ family}.
\end{proof}

\begin{corollary}\label{prop-univ-family-Hfib}
  For any \(n\in \N\), the morphism \(p^{\stGH}_n:\stGH_{n+1}\to \stGH_{n}\) of stacks in groupoids over \(\tbMet_{\lhftop}\) is a representable proper \Hfib \confer{\cref{def-Hfib-for-stacks}}.
  In particular, for any \(n\in\N\), \(\stGH_n\) is a naive metric stack of quotient type.
\end{corollary}

\begin{proof}
  \Cref{prop-univ-family-Hfib} follows immediately from \cref{prop-naive-metric-of-quot-and-rep-proper-Hfib} and \cref{lem: univ family}, together with the 2-Yoneda lemma.
\end{proof}

Finally, we prove the statement asserting that \(\stGH\) is a ``fine moduli'' of compact metric spaces with universal object \(p_0^{\stGH}:\stGH_1\to \stGH = \stGH_0\). 

\newcommand{\SmRP}{\mathsf{Sm}_{rp}}

\begin{theorem}\label{prop-fine-moduli-GH}
  Let \(\mcX\) be a naive metric stack of quotient type. 
  Write \(\SmRP(\mcX)\) for the 2-groupoid defined as follows:
  \begin{itemize}
    \item An object of \(\SmRP(\mcX)\) is a representable proper \Hfib \(p:\mcP\to \mcX\).
    \item 
    A morphism in \(\SmRP(\mcX)\) from an object \(p:\mcP\to \mcX\) to \(p':\mcP'\to \mcX\) is a pair \((f, \rho)\) where \(f:\mcP\to \mcP'\) is an equivalence and \(\rho:p'\circ f\tosim p\) is a 2-morphism.
    \item 
    A 2-morphism in \(\SmRP(\mcX)\) from a morphism \((f:p\to p', \rho:p'\circ f\tosim p)\) to a morphism \((f':p\to p', \rho':p'\circ f'\tosim p)\) is a 2-morphism \(\alpha:f\tosim f'\) such that \(\rho'\circ (p'\star \alpha) = \rho\). 
  \end{itemize}
  We regard the groupoid \(\Hom_{\Stack_{\lhftop}(\tbMet)}(\mcX, \stGH)\) as a 2-groupoid with only identity 2-morphisms. 
  Then, by taking the 2-pullback of \(p_0^{\stGH}:\stGH_1\to \stGH_0=\stGH\), we obtain a functorial (with respect to \(\mcX\)) equivalence of 2-groupoids 
  \[\Psi_{\mcX}:\Hom_{\Stack_{\lhftop}(\tbMet)}(\mcX, \stGH)\tosim \SmRP(\mcX).\]
\end{theorem}

\begin{proof}
  The functoriality of \(\Psi_{\mcX}\) follows immediately from the construction. 
  Moreover, by the universality of the 2-fiber product, the 2-functor \(\Psi_{\mcX}\) is fully faithful in the 2-categorical sense. 
  Hence, to prove \cref{prop-fine-moduli-GH}, it suffices to prove that the 2-functor \(\Psi_{\mcX}\) is essentially surjective in the 2-categorical sense (note that, once this is established, it follows that any 2-morphism in \(\SmRP(\mcX)\) between given morphisms is unique).
  Let \(p:\mcP\to \mcX\) be a representable proper \Hfib. 
  Since \(\mcX\) is of quotient type, there exist a metric space \(U\in\Met\) and a \fHfib \(u:U\to \mcX\) such that the induced morphism \([U/\mcR]\to \mcX\), where \(\mcR\dfn U\times_{\mcX}U\), is an equivalence. 
  Since \(p\) is a representable proper \Hfib, it follows from \cref{prop-rep-proper-Hfib-in-met} that there exist a 2-Cartesian square
  \[\begin{tikzcd}
    {Q} \ar[r, "{v}"] \ar[d, "{q}"'] & 
    {\mcP} \ar[d, "{p}"] \ar[dl, Rightarrow, "{\alpha}"', shorten=3mm] \\ 
    {U} \ar[r, "{u}"] & {\mcX}
  \end{tikzcd}\]
  such that \(q:Q\to U\) is a proper \Hfib in \(\Met\). 
  Then, by the proof of \cref{prop-rep-proper-Hfib-in-met}, the natural morphism \([Q/\mcS]\to \mcP\), where \(\mcS\dfn Q\times_{\mcP}Q\), is an equivalence. 

  Write 
  \begin{itemize}
    \item \([q]:U\to \stGH\) for the morphism of stacks in groupoids over \(\tbMet_{\lhftop}\) defined as follows: for each object \(T\in \tbMet\), \([q]_T(a:T\to U)\dfn [Q\times_UT\to T]\in \stGH(T)\);
    \item \(\pr_1, \pr_2:\mcR\tto U\) for the natural projections; \(\eta:u\circ \pr_1\tosim u\circ \pr_2\) for the natural 2-morphism;
    \item \(\tilde{q}_1:\mcQ_1\dfn Q\times_{U,\pr_1}\mcR\to \mcR\) and \(\tilde{q}_2:\mcQ_2\dfn Q\times_{U,\pr_2}\mcR\to \mcR\) for the natural projections; 
    \item \(r_1:\mcQ_1\to Q\) and \(r_2:\mcQ_2\to Q\) for the natural projections.
  \end{itemize}
  Since \(\mcQ_2\) is a 2-fiber product of the diagram \(\mcR\xto{u\circ \pr_2}\mcX\xgets{p}\mcP\), the triple \((\tilde{q}_1:\mcQ_1\to \mcR, v\circ r_1:\mcQ_1\to \mcP, (\tilde{q}_1\star\eta)\circ (\alpha\star r_1):q\circ v\circ r_1\tosim u\circ \pr_2\circ \tilde{q}_1)\) induces an isomorphism of sheaves \(\alpha:\mcQ_1\to \mcQ_2\) and a 2-morphism \(\beta:v\circ r_2\circ \alpha\tosim v\circ r_1\). 
  Then, \(\alpha\) induces a 2-morphism \([\alpha]: [q]\circ \pr_1\tosim [q]\circ \pr_2\). 
  Since \([U/\mcR]\tosim \mcX\), the triple \(([q]\circ \pr_1, [q]\circ \pr_2, [\alpha])\) induces a morphism of stacks in groupoids \([\mcP]:\mcX\to \stGH\). 
  Moreover, by construction, the triple \((r_1, r_2\circ \alpha, \beta)\) induces an isomorphism of sheaves \(\mcQ_1\tosim \mcS\). 
  Hence, we may regard \(\mcS=\mcQ_1\).
  Write \([\Delta_{Q/U}]:Q\to \stGH_1\) for the morphism of stacks in groupoids over \(\tbMet_{\lhftop}\) defined as follows: for each object \(T\in\tbMet_{\lhftop}\), \[[\Delta_{Q/U}]_T(a:T\to Q)\dfn [(Q\times_UT\to T, (a,\id_T): T\to Q\times_UT)]\in \stGH_1(T).\]
  Note that the isomorphism of sheaves \(\alpha:\mcQ_1\tosim \mcQ_2\) induces, via the natural isomorphisms \(\mcQ_1\times_{q\circ r_1, U}Q\tosim \mcQ_1\times_{\tilde{q}_1,\mcR,\tilde{q}_1}\mcQ_1\) and \(\mcQ_1\times_{q\circ r_2\circ \alpha, U}Q\tosim \mcQ_1\times_{\tilde{q}_1,\mcR,\tilde{q}_2}\mcQ_2\), an isomorphism \(\alpha_1:\mcQ_1\times_{q\circ r_1, U}Q\tosim \mcQ_1\times_{q\circ r_2\circ \alpha, U}Q\) over \(\mcQ_1\) such that \(\alpha_1\circ \Gamma_{r_1} = \Gamma_{r_2\circ \alpha}\), where \(\Gamma_{r_1}:\mcQ_1\to \mcQ_1\times_{q\circ r_1,U}Q\) and \(\Gamma_{r_2\circ \alpha}:\mcQ_1\to\mcQ_1\times_{q\circ r_2\circ \alpha,U}Q\) are the graph of morphisms \(r_1, r_2\circ \alpha\) over \(U\). 
  Hence, the isomorphism \(\alpha_1\) induces a 2-morphism \([\Delta_{Q/U}]\circ r_1\tosim [\Delta_{Q/U}]\circ r_2\circ \alpha\) and a morphism \([\Delta_{\mcP/\mcX}]:\mcP\to \stGH_1\) of stacks in groupoids over \(\tbMet_{\lhftop}\). 
  Finally, by these constructions, the diagram 
  \[\begin{tikzcd}
    {\mcP} \ar[r, "{[\Delta_{\mcP/\mcX}]}"] \ar[d, "{p}"'] &[5mm] {\stGH_1} \ar[d, "{p_0^{\stGH}}"] \\
    {\mcX} \ar[r, "{[\mcP]}"] & {\stGH}
  \end{tikzcd}\]
  is 2-Cartesian.
  This completes the proof of \cref{prop-fine-moduli-GH}.
\end{proof}

\end{document}